\newcommand{\si}[1]{#1}    
\newcommand{\jo}[1]{}    
\newcommand{\ima}[1]{}   

\si{
	\documentclass[10pt,a4paper]{article}
	\usepackage[utf8]{inputenc}
	\usepackage{graphicx}
	\usepackage{amsfonts}
	\usepackage{amsthm}
	\usepackage{amsmath}
	\usepackage{amssymb}
	\usepackage{array}
	\usepackage{xcolor}
	\usepackage{booktabs}
	\usepackage{algorithm}
	\usepackage{indentfirst}
	\usepackage{algorithmic}
	\usepackage{hyperref}
	\usepackage[left=2cm,right=2cm,top=2cm,bottom=2cm]{geometry}
	
	\tracingstats=0
	
	\newtheorem{theorem}{Theorem}[section]
	\newtheorem{proposition}{Proposition}[section]
	\newtheorem{corollary}{Corollary}[section]
	\newtheorem{lemma}{Lemma}[section]
	\newtheorem{definition}{Definition}[section]
	\newtheorem{example}{Example}[section]
	\newtheorem{remark}{Remark}[section]

	\hypersetup{
	  colorlinks=true,
	  citecolor=blue,
	  linkcolor=blue,
	  filecolor=magenta,      
	  urlcolor=cyan,
	}
}

\jo{
	\documentclass[smallextended,referee,envcountsect,]{svjour3}

	\smartqed					
	\usepackage{graphicx}		
	\usepackage{amsfonts}
	\usepackage{amsmath}
	\usepackage{amssymb}
	\usepackage{array}
	\usepackage{xcolor}
	\usepackage{booktabs}
	\usepackage{algorithm}
	\usepackage{indentfirst}
	\usepackage{algorithmic}
	\usepackage{hyperref}
	\journalname{JOTA}
}

\ima{
	\documentclass{imanum}
	\usepackage{graphicx}
	\usepackage{array}
	\usepackage{booktabs}
	\usepackage{algorithm}
	\usepackage{indentfirst}
	\usepackage{algorithmic}
	\jno{drnxxx}
	\received{22 October 2021}
}

\usepackage{tikz}

\newcommand{\E}{\mathcal{E}}					 
\renewcommand{\S}{\mathbb{S}}					 

\newcommand{\seq}[1]{\{#1^k\}_{k\in \N}}		 
\newcommand{\argmin}{\mathop{\mathrm{argmin}}}   
\renewcommand{\bar}{\overline}                   
\newcommand{\bd}{\mathrm{bd\hspace{0.03cm}}}                    
\newcommand{\cl}{\mathrm{cl\hspace{0.03cm}}}  
\newcommand{\I}{\mathbb{I}}                      
\newcommand{\Ker}{\mathrm{Ker\hspace{0.03cm}}}   
\renewcommand{\Im}{\mathrm{Im\hspace{0.03cm}}}   
\renewcommand{\int}{\mathrm{int\hspace{0.03cm}}} 
\newcommand{\F}{\mathcal{F}}                     
\newcommand{\norm}[1]{\|#1\|}                    
\newcommand{\enorm}[1]{\|#1\|_2}                 
\newcommand{\R}{\mathbb{R}}  					 
\newcommand{\Lor}{\mathbb{L}}                    
\newcommand{\T}{\top\hspace{-1pt}}               
\newcommand{\N}{\mathbb{N}}                      
\newcommand{\xb}{\bar{x}}                        

\newcommand{\mand}{ \ \textnormal{ and } \ }
\newcommand{\sdpwcrcq}{weak-CRCQ}
\newcommand{\sdpwcpld}{weak-CPLD}
\newcommand{\socpwcrcq}{weak-CRCQ}
\newcommand{\socpwcpld}{weak-CPLD}

\newcommand{\socpscrcq}{seq-CRCQ}
\newcommand{\socpscpld}{seq-CPLD}
\newcommand{\nlpcrcq}{CRCQ}
\newcommand{\nlpcpld}{CPLD}

\renewcommand{\doteq}{:=}

\si{
}

\begin{document}


\newcommand{\acknowledgementstext}{
	The authors received financial support from FAPESP (grants \jo{\linebreak}2017/18308-2, 2017/17840-2, \ima{\linebreak}2017/12187-9, 2018/24293-0 and 2020/00130-5), \jo{\linebreak}CEPID-CeMEAI (granted by FAPESP 2013/07375-0), CNPq (grants  301888/2017-5, \jo{\linebreak}303427/2018-3, 404656/2018-8, and 306988/2021-6), PRONEX - CNPq/FAPERJ (grant \jo{\linebreak}E-26/010.001247/2016), and ANID (FONDECYT grant  1201982, Program  ANID \jo{\linebreak}ACE210010 and  Basal Program CMM ANID  AFB210005). \jo{Data sharing is not applicable to this article as no new data were created or analysed in this study.}
}

	\title{Global Convergence of Algorithms Under Constant Rank Conditions for Nonlinear Second-Order Cone Programming\si{\footnotetext{\acknowledgementstext}}}

\si{
	\author{
	Roberto Andreani \thanks{Department of Applied Mathematics, University of Campinas, Campinas, SP, Brazil. 
		Email: {\tt andreani@unicamp.br}}
	\and
	Gabriel Haeser \thanks{Department of Applied Mathematics, University of S{\~a}o Paulo, S{\~a}o Paulo, SP, Brazil. 
		Emails: {\tt \{ghaeser,leokoto,thiagops\}@ime.usp.br}}
	\and 
	Leonardo M. Mito \footnotemark[2]
	\and
	H{\'e}ctor Ram{\'i}rez C. \thanks{Departamento de Ingenier{\'i}a Matem{\'a}tica and Centro de Modelamiento Matem{\'a}tico (CNRS IRL2807), Universidad de Chile, Santiago, Chile.
	Email: {\tt hramirez@dim.uchile.cl}}
	\and
	Thiago P. Silveira \footnotemark[2]
	}
}

\jo{
	\titlerunning{Global Convergence of Algorithms Under Constant Rank Conditions for NSOCP}


	\author{
	Roberto Andreani 
	\and
	Gabriel Haeser
	\and
	Leonardo M. Mito
	\and		
	H{\'e}ctor Ram{\'i}rez C.
	\and
	Thiago P. Silveira
	}

	\institute{
				Roberto Andreani \at
              	Department of Applied Mathematics, University of Campinas \\
              	Campinas, SP, Brazil\\
             	 andreani@unicamp.br
           \and
              	Gabriel Haeser \at
              	Department of Applied Mathematics, University of S{\~a}o Paulo \\
              	S{\~a}o Paulo, SP, Brazil \\
              	ghaeser@ime.usp.br
		   \and
              	Leonardo M. Mito \at
              	Department of Applied Mathematics, University of S{\~a}o Paulo \\
              	S{\~a}o Paulo, SP, Brazil \\
              	leokoto@ime.usp.br    
           \and
              	H{\'e}ctor Ram{\'i}rez C. \at
              	Departamento de Ingenier{\'i}a Matem{\'a}tica and Centro de Modelamiento Matem{\'a}tico, Universidad de Chile \\
              	Santiago, Chile \\
              	hramirez@dim.uchile.cl
           \and
              	Thiago P. Silveira \at
              	Department of Applied Mathematics, University of S{\~a}o Paulo \\
              	S{\~a}o Paulo, SP, Brazil \\
              	thiagops@ime.usp.br
	}

\date{Received: date / Accepted: date}

}

\ima{
	\shorttitle{Global Convergence of Algorithms Under Constant Rank Conditions for NSOCP}

\author{%
{\sc Roberto Andreani}\thanks{Email: andreani@unicamp.br}\\[2pt]
Department of Applied Mathematics, University of Campinas, Rua S{\'e}rgio Buarque de Holanda, 651, 13083-859, Campinas, SP, Brazil.\\[6pt]
{\sc
Gabriel Haeser\thanks{Email: ghaeser@ime.usp.br},
Leonardo M. Mito\thanks{Corresponding author. Email: leokoto@ime.usp.br} 
and 
Thiago P. Silveira\thanks{Email: thiagops@ime.usp.br}
}\\[2pt]
Department of Applied Mathematics, University of S{\~a}o Paulo, Rua do Mat{\~a}o, 1010, Cidade Universit{\'a}ria, 05509-090, S{\~a}o Paulo, SP, Brazil.\\[6pt]
{\sc and}\\[6pt]
{\sc H{\'e}ctor Ram{\'i}rez C.}\thanks{Email: hramirez@dim.uchile.cl}\\[2pt]
Departamento de Ingenier{\'i}a Matem{\'a}tica and Centro de Modelamiento Matem{\'a}tico (CNRS IRL2807),  Universidad de Chile, Beauchef 851, Santiago, Chile.
}
	
\shortauthorlist{R. Andreani \emph{et al.}}
}


\newcommand{\abstracttext}{
	In [R.~Andreani, G.~Haeser, L.~M. Mito, H.~Ram{\'i}rez~C., Weak notions of nondegeneracy in nonlinear semidefinite programming, \jo{\linebreak}arXiv:2012.14810, 2020] the classical notion of nondegeneracy (or transversality) and Robinson's constraint qualification have been revisited in the context of nonlinear semidefinite programming exploiting the structure of the problem, namely, its eigendecomposition. This allows formulating the conditions equivalently in terms of (positive) linear independence of significantly smaller sets of vectors. In this paper we extend these ideas to the context of nonlinear second-order cone programming. For instance, for an $m$-dimensional second-order cone, instead of stating nondegeneracy at the vertex as the linear independence of $m$ derivative vectors, we do it in terms of several statements of linear independence of $2$ derivative vectors. This allows embedding the structure of the second-order cone into the formulation of nondegeneracy and, by extension, Robinson's constraint qualification as well. This point of view is shown to be crucial in defining significantly weaker constraint qualifications such as the constant rank constraint qualification and the constant positive linear dependence condition. Also, these conditions are shown to be sufficient for guaranteeing global convergence of several algorithms, while still implying metric subregularity and without requiring boundedness of the set of Lagrange multipliers.
}

\maketitle

\si{
	\abstract{
		\abstracttext
		
		\
		
	}
}

\jo{
	\begin{abstract}
		\abstracttext
	\end{abstract}
}

\ima{
	\begin{abstract}
		{\abstracttext}
		{Second-order cone programming, Constraint qualifications, Algorithms, Global convergence, Constant rank.}
	\end{abstract}
}

%
%
%
%
%
%
%
%
%
%
%
%

\section{Introduction}\label{sec:intro}

The well-known \textit{constant rank constraint qualification} (CRCQ) was introduced by \si{Janin}\jo{Janin}~\cite{janin}, for \textit{nonlinear programming} (NLP), with the purpose of obtaining a formula for the Hadamard directional derivative of the value function. Prior to his work, similar results wheren known under the \textit{Mangasarian-Fromovitz constraint qualification} (MFCQ)\ima{ (see}~\cite{gauvindubeau,rockavalue}\ima{)} and the \textit{linear independence constraint qualification} (LICQ)\ima{ (see also}~\cite{gauvindubeau}\ima{)}.

Janin also showed that CRCQ neither implies nor is implied by MFCQ and, moreover, that CRCQ is strictly weaker than LICQ. After that, CRCQ has been widely employed in the NLP literature for instance in the study of stability \ima{(}\cite{outratampec1,tilt,outrata1}\ima{)}, strong second-order necessary optimality conditions \cite{aes2}, global convergence of algorithms \ima{(}\cite{abms}\ima{)}, among other applications. We remark that CRCQ explains in a very simple way the existence of Lagrange multipliers associated with affine constraints, such as in linear programming.

More recently, \si{Qi and Wei}\jo{Qi and Wei}~\cite{qiwei} presented a condition called \textit{constant positive linear dependence} (CPLD), which is strictly weaker than both MFCQ and CRCQ, and showed its application on the convergence of a general \textit{sequential quadratic programming} (SQP) method for NLP. However, they did not prove that CPLD was a constraint qualification at the time. This issue was settled in a later article by \si{Andreani et al.}\jo{Andreani et al.}~\cite{Andreani2005}, where they proved that CPLD implies the \textit{quasinormality} constraint qualification condition. Later, in~\cite{abms}, the convergence of an \textit{augmented Lagrangian} method was also proved under CPLD. Other uses of constant rank-type constraint qualifications in NLP are discussed, for instance, in~\cite{ahss12,Andreani2012,janin,rcr,param} and their references. In particular, the appropriate way of incorporating equality constraints in the definitions of CRCQ and CPLD are discussed respectively in \cite{rcr} and \cite{ahss12}.

Although constraint qualifications with applications towards convergence of algorithms are largely studied in NLP, the situation is quite different in \textit{nonlinear second-order cone programming} (NSOCP), despite its many relevant applications -- for example, in structural optimization and machine learning~\ima{ (}\cite{sosvm}\ima{)}, hydroacoustic classification of fishes~\ima{ (}\cite{fishes}\ima{)}, and others~\ima{ (}\cite{boydsocsurvey}\ima{)}. In NSOCP, this role is almost always covered by the so-called \textit{nondegeneracy} condition (c.f. \cite[Equation 25]{bonnansramirez}) and \textit{Robinson's constraint qualification} (Robinson's CQ) (c.f. \cite[Equation 29]{bonnansramirez}), which can be seen as natural generalizations of LICQ and MFCQ, respectively. The first work that attempted to extend CRCQ and its variants to the context of NSOCP is due to \si{Zhang and Zhang}\jo{Zhang and Zhang}~\cite{ZZ}, but their condition was invalidated by a counter-example given by \si{Andreani et al. in}\jo{Andreani et al. in}~\cite{ZZerratum}. Later, a ``naive approach'' to extend some constant rank-type constraint qualifications for NSOCP was presented by \si{Andreani et al. in}\jo{Andreani et al. in}~\cite{crcq-naive}; the adjective ``naive'' refers to the fact that some of the conic constraints were locally rewritten as NLP constraints whenever possible, yielding a new reformulated problem with mixed constraints, and then a hybrid condition between the NLP versions of CRCQ/CPLD and nondegeneracy/Robinson's CQ was presented. The major contribution of~\cite{crcq-naive} is to show an effective way of dealing with those two distinct types of constraints via sequences of approximate stationary points.

Recently, we proposed in \cite{facial-crcq} a new geometrical characterization of CRCQ for NLP using the faces of the non-negative orthant, which was naturally extended to the context of NSOCP as well as \textit{nonlinear semidefinite programming} (NSDP). This has led us to an alternative constant rank-type constraint qualification that allowed us to derive strong second order optimality conditions for NSDP and NSOCP without assuming compactness of the Lagrange multiplier set, similarly to what is known in NLP~\ima{(}\cite{aes2}\ima{)}. However, no application towards algorithms was provided or suggested in~\cite{facial-crcq}. Since the sequential approach from~\cite{crcq-naive} seems more suitable for algorithms, we developed it even further for NSDP problems\ima{ in}~\cite{seq-crcq-sdp,weaksparsecq} by directly exploiting the eigenvector structure of the problem, overcoming the limitations of the naive approach.

This paper introduces new constraint qualifications for NSOCP problems following similar ideas to those used in~\cite{seq-crcq-sdp} and~\cite{weaksparsecq}, but taking into account the structure of the second-order cone. For such, we will first introduce weak variants of the nondegeneracy condition and Robinson's CQ -- here called \textit{weak-nondegeneracy} and \textit{weak-Robinson's CQ} -- which are weaker than their original versions but that still reduce to LICQ and MFCQ, respectively, when an NLP problem is modelled as an NSOCP problem with several one-dimensional constraints. Moreover, we show that weak-nondegeneracy is strictly weaker than nondegeneracy, and we also clarify some relations between weak-nondegeneracy (weak-Robinson's CQ) and standard nondegeneracy (Robinson's CQ), which were only partially addressed in~\cite{seq-crcq-sdp}. In particular, we show a new characterization of nondegeneracy in terms of the validity of weak-nondegeneracy plus the linear independence of a partial Jacobian of the constraints. The relationship of weak-Robinson's CQ and Robinson's CQ is also partially settled in our Theorem \ref{socp:wrob_rob_sep}, which was left as an open problem for NSDP in \cite{weaksparsecq}. With these new constraint qualifications at hand, we introduce new extensions of CRCQ and CPLD for NSOCP, which also recover their counterparts in NLP. We also discuss a mild adaptation of these new conditions that can be adopted with the purpose of proving global convergence results for algorithms that keep track of Lagrange multiplier estimates. 

The structure of this paper is as follows: in Section~\ref{sec:prelim} we present some notation and technical results. Sections~\ref{sec:socp} and~\ref{sec:cr} present weak constraint qualifications for NSOCP: weak-nondegeneracy condition, weak-Robinson's CQ, and two weak constant rank conditions. Also, we present some of their properties and a detailed comparison with other constraint qualifications from the literature, and among themselves. In Section~\ref{sec:appl} we introduce perturbed versions of the constant rank conditions of Section~\ref{sec:cr}, and we present some algorithms related to them. We state the relationship between these perturbed variants and the so-called \textit{metric subregularity CQ}. Finally, in Section~\ref{sec:conclusion}, we summarize our results and discuss some ideas for future research.


\section{Preliminaries}\label{sec:prelim}

In this section, we will present our notation, and some linear algebra and convex analysis tools needed for deriving the results of this paper.

\subsection{Basic Results and Some Notation}

For a given differentiable function $F\colon \R^n\to \R^m$, we denote the \textit{Jacobian} matrix of $F$ at a point $x\in \R^n$ by $DF(x)$; and the ${j}$-th column of its \textit{transpose}, $DF(x)^\T$, will be denoted by $\nabla F_{j}(x)$. We also adopt the usual inner product in $\R^m$, given by $\langle y,z\rangle\doteq \sum_{{j}=1}^m y_{j} z_{j}$, along with the \textit{Euclidean norm} $\|y\|\doteq \sqrt{\langle y,y\rangle}$, for every $y,z\in \R^m$. The open ball (respective to the Euclidean norm) that has center at $y$ and radius $\delta\geq 0$ will be denoted by $B(y,\delta)$, and its closure, by $\cl(B(y,\delta))$. 

The orthogonal projection of a given $y\in\R^m$ onto a nonempty closed convex set $C\subseteq \R^m$ with respect to the Euclidean norm is defined as
\[
	\mathcal{P}_C(y)\doteq \argmin_{z\in C} \|z-y\|.
\]
It is valid to mention that $\mathcal{P}_C(y)$ is well-defined as a continuous function of $y$, since $C$ is closed and convex. Also, when $C$ is given by the Cartesian product of other non-empty closed convex sets $C_1,\ldots,C_q$, where $C_{j}\subseteq \R^{m_{j}}$ for every ${j}\in \{1,\ldots,q\}$, then for any $y\doteq(y_1,\ldots,y_q)\in \R^{m_1+\ldots+m_q},$ we have
\[
	\mathcal{P}_C(y)=\left( \mathcal{P}_{C_1}(y_1), \ldots, \mathcal{P}_{C_q}(y_q) \right).
\]

To relate our results with the classical ones from the literature, we will make use of a notion of \textit{conic linear independence}, defined as follows:
\begin{definition}\label{socp:conicli}
Let $C\subseteq \R^{m}$ be a nonempty closed convex cone. A matrix $M\in \R^{n\times m}$ is said to be \emph{$C$-linearly independent} if there is no non-zero $v\in C$ such that $Mv=0$.  
\end{definition}

Roughly speaking, Definition~\ref{socp:conicli} describes ``injectivity over $C$''. In particular, if $C$ is the nonnegative orthant \[\R^m_+\doteq \{y\in \R^m\colon \forall i\in \{1,\ldots,m\}, \ y_i\geq 0\},\]
then Definition~\ref{socp:conicli} reduces to a concept known in NLP as \textit{positive linear independence} of the columns of $M$. Now, let us show a simple characterization of conic linear independence in terms of all finitely generated conical slices of the cone. 

\begin{lemma}\label{lem:conicli}
Let $C\subseteq \R^m$ be a closed convex cone such that there exists a (possibly infinite) index set $S$ and, for each $w\in S$, a finite subset $\mathcal{E}_{w}\subseteq C$ whose elements are linearly independent, such that
\begin{equation}\label{eq:conicslices}
	C=\bigcup_{w\in S} \textnormal{cone}(\mathcal{E}_{w}),
\end{equation}
where $\textnormal{cone}(\E_{w})$ denotes the conic hull of $\E_{w}$. Then, a matrix $M\in \R^{n\times m}$ is $C$-linearly independent if, and only if, the family $\{Mv\}_{v\in \E_w}$ is positively linearly independent, for every fixed $w\in S$.
\end{lemma}

\begin{proof}
Suppose that $M$ is $C$-linearly independent, let $w\in S$ be arbitrary, and let $a_{v}\in \R_+$, $v\in \E_{w}$, be scalars such that
\begin{equation}\label{aux:eq1}
	\sum_{v\in \E_{w}}  a_{v}M v=M\left[\sum_{v\in \E_{w}} a_{v} v\right]=0.
\end{equation}
Since $C$ is a convex cone, it follows that $\tilde{v}\doteq \sum_{v\in \E_{w}} a_{v} v$ belongs to $C$, so $\tilde{v}=0$ by hypothesis; and from the linear independence of $\E_{w}$ we have that $a_{v}=0$ for every $v\in \E_{w}$. Thus, $\{Mv\}_{v\in \E_w}$ is positively linearly independent. 

Conversely, assume that $\{Mv\}_{v\in \E_w}$ is positively linearly independent, and let $\tilde{v}\in C$ be such that $M\tilde{v}=0$. Then, there is some $w\in S$ such that $\tilde{v}\in \textnormal{cone}(\E_{w})$; that is, there exist some scalars $a_{v}\geq 0$, $v\in \E_w$, such that $\tilde{v}=\sum_{v\in \E_{w}} a_{v}v$ and hence~\eqref{aux:eq1} holds, implying that $a_{v}=0$ for all $v\in \E_w$; thus $\tilde{v}=0$.
\jo{\qed} \end{proof}

\begin{remark}\label{remarklemma}Considering $C=\R^m$ in the statement of the Lemma and replacing the conic hull by the linear span in \eqref{eq:conicslices}, we arrive similarly at a characterization of the linear independence of the columns of $M$ in terms of the linear independence of the family $\{Mv\}_{v\in\mathcal{E}_w}$, for every fixed $w\in S$.\end{remark}

A simple example to fix ideas on how to use Lemma~\ref{lem:conicli} is to take the parametric representation of $\R^2$:
\begin{equation}\label{eq:r2circle}
	\begin{aligned}
	\R^2 & =\{(r\cos(w),r\sin(w))\colon w\in [0,2\pi], r\geq 0\} \jo{\\ &} =\bigcup_{w\in [0,2\pi]} \textnormal{cone}((\cos(w),\sin(w)))
	\end{aligned}
\end{equation}
so we have $C=\R^2$, $S=[0,2\pi]$, and $\E_w=\{(\cos(w),\sin(w))\}, w\in S$. In this case Lemma~\ref{lem:conicli} simply states the trivial fact that a matrix $M\in \R^{n\times 2}$ is injective if, and only if, $M(\cos(w),\sin(w))\neq 0$ for every $w\in [0,2\pi]$. Moreover, the main object of our study, the \textit{second-order cone} (or \textit{Lorentz cone}):
\[
	{\Lor_{m}}\doteq \left\{
		\begin{array}{ll}
			\{y\doteq(y_0,\widehat{y})\in \R\times \R^{m-1}\colon y_0\geq \norm{\widehat{y}}\}, & \textnormal{ if } m>1,\\
			\R_+, & \textnormal{ if } m=1,
		\end{array}
	\right.
\]
may benefit from Lemma~\ref{lem:conicli} as well, since it can be written as
\[
	\Lor_m=\bigcup_{\substack{w\in\R^{m-1}\\ \|w\|=1}} \textnormal{cone}(\{(1,-w),(1,w)\}),
\]
which corresponds to $S=\{w\in\R^{m-1}\colon \|w\|=1\}$ and $\E_w=\{(1,-w),(1,w)\}$. In this case Lemma~\ref{lem:conicli} states that a matrix $M\in \R^{n\times m}$ is $\Lor_m$-linearly independent if, and only if, the vectors
\begin{equation}\label{eq:vecexpli}
	M(1,-w) \quad\textnormal{and}\quad M(1,w)
\end{equation}
are positively linearly independent for every $w\in \R^{m-1}$ such that $\|w\|=1$. Furthermore, the standard notion of linear independence in $\R^m$ can also be stated in terms of the conical slices of $\Lor_m$, since it is a full-dimensional cone; indeed, observe that
\[
	\R^m=\bigcup_{\substack{w\in\R^{m-1}\\ \|w\|=1}} \textnormal{span}(\{(1,-w),(1,w)\}),
\]
where $\textnormal{span}(\{(1,-w),(1,w)\})$ denotes the \textit{linear span} of the vectors $(1,-w)$ and $(1,w)$; then, the matrix $M$ is $\R^m$-linearly independent (i.e., injective) if, and only if, the vectors~\eqref{eq:vecexpli} are linearly independent for every $w\in \R^{m-1}$ such that $\|w\|=1$. Thus, we have replaced the linear independence of the $m$ columns of $M$ by a series of linear independence requirements of only $2$ parameterized vectors in \eqref{eq:vecexpli}, independently of the size of $m$. With this point of view, we will be able to exploit the structure of the second-order cone, which will turn out to be essential in our analysis.

Furthermore, observe that Lemma~\ref{lem:conicli} can be applied to products of closed convex cones $C=\prod_{j\in J} C_j$, where $J$ is an index set, in order to describe $C$-linear independence of a family of matrices $\{M_j\}_{j\in J}$ mounted into a conveniently indexed block matrix
\begin{equation}\label{eq:Mj}
M \doteq \left[ \begin{array}{ccc}
	& \vdots &\\
     & M_j & \\ 
     & \vdots &
\end{array} \right]_{{j} \in J} 
\end{equation}
therefore, we will abuse the terminology to define the $C$-linear independence of the family $\{M_j\}_{j\in J}$ in terms of the above $M$ throughout the paper.

To close this subsection, let us briefly recall the celebrated \emph{Carath{\'e}odory's Lemma}~\cite[Exercise B.1.7]{bertsekasnl} from convex analysis:

\begin{lemma}[Carath{\'e}odory's Lemma]\label{lem:carath} Let $y_1,\dots,y_{p}\in\mathbb{R}^n$, and let \si{$\alpha_1,\ldots,\alpha_{p}\in \R$} \jo{$\alpha_1, \linebreak \ldots,\alpha_{p}\in \R$} be arbitrary. Then,
there exists some $J \subseteq  \{1, \ldots, p \}$ and some scalars $\tilde{\alpha}_{j}$ with ${j}\in J$, such that $ \{y_{j} \}_{{j}\in J} $ is linearly independent,
\[
	 \sum_{{j}=1}^{p} \alpha_{j} y_{j} = \sum_{{j}\in J} \tilde{\alpha}_{j} y_{j},
\]
and $\alpha_{j}\tilde{\alpha}_{j}>0$, for all ${j}\in J$.
\end{lemma}

\subsection{The Nonlinear Second-Order Cone Programming Problem}

A (multifold) nonlinear second-order cone programming problem is usually stated in the form:
\begin{equation}\label{NSOCP}
  \tag{NSOCP}
  \begin{aligned}
    & \underset{x \in \mathbb{R}^{n}}{\text{Minimize}}
    & & f(x), \\
    & \text{subject to}
    & & g_{j}(x) \in  {\Lor_{m_{j}}}, \ \forall {j}\in\{1,\ldots,q\},\\
  \end{aligned}
\end{equation}
where $f\colon \R^n\to \R$ and $g_{j}\colon \R^n\to \R^{m_{j}}$ are continuously differentiable functions, for all ${j}\in \{1,\ldots,q\}$, and $\Lor_{m_j}$ is a second-order cone in $\R^{m_{j}}$. As usual, for a point $x\in\R^n$ we denote $g_j(x)=(g_{j,0}(x),\widehat{g}_j(x))\in\R\times\R^{m_j-1}$. The feasible set of (\ref{NSOCP}) will be denoted by $\F$. Also, we denote the \textit{interior} and the \textit{boundary excluding the origin} of $\Lor_{m_{j}}$ by $\int \Lor_{m_{j}}$ and $\bd_+ \Lor_{m_{j}}$, respectively; and as usual in the study of NSOCP, for any $x\in \F$ we partition $\{1,\ldots,q\}$ as follows:
\begin{equation}
  \label{socp:index}
  \begin{array}{r@{\:}c@{\:}l}
    I_0(x) & \doteq   & \{ {j}\in \{1,\ldots,q\} \colon 
    g_{j}(x) = 0 \}, \\
    I_B(x) & \doteq   & \{{j}\in \{1,\ldots,q\} \colon 
    g_{j}(x) \in \bd_+{\Lor_{m_{j}}} \},\\
    I_{\int}(x) & \doteq   & \{ {j}\in \{1,\ldots,q\} \colon 
    g_{j}(x) \in \int {\Lor_{m_{j}}} \}.
  \end{array}
\end{equation}

Following~\cite[Section 4]{surveysocp}, we recall that if $m_{j}>1$, then every $y\in\R^{m_{j}}$ has a \textit{spectral decomposition} with respect to ${\Lor_{m_{j}}}$, in the form
\[
	y=\lambda_1(y)u_1(y)+\lambda_2(y)u_2(y),
\]
where
\begin{equation}\label{socp:autovals}
	\lambda_i(y)\doteq y_0+(-1)^i\norm{\widehat{y}} \quad \mand \quad u_i(y)\doteq \left\{
	\begin{aligned}
		\frac{1}{2}\left(1,(-1)^i\frac{\widehat{y}}{\norm{\widehat{y}}}\right), & \quad \textnormal{ if } \widehat{y}\neq 0,\\
		\frac{1}{2}\left(1,(-1)^i w\right), & \quad \textnormal{ otherwise},
	\end{aligned}
	\right.
\end{equation}
and $w\in \R^{m_{j}-1}$ can be any unitary vector, with $i\in \{1,2\}$. In this setting, $\lambda_i(y)$ is said to be an eigenvalue of $y$ associated with the eigenvector $u_i(y)$, $i\in \{1,2\}$. By definition, we see that $y \in \Lor_{m_{{j}}}$ if, and only if, $\lambda_{1}(y)\geq 0, \lambda_{2}(y) \geq 0$, whence follows that the orthogonal projection of $y$ onto $\Lor_{m_{j}}$ can be characterized as
\[
	\mathcal{P}_{{\Lor_{m_{j}}}}(y)=[\lambda_1(y)]_+u_1(y) + [\lambda_2(y)]_+u_2(y),
\]
where $[ \ \cdot \ ]_+\doteq \max\{ \ \cdot \ , 0\}$.
\begin{remark}\label{socp:remark0}
From this point onwards, we will assume that $m_{j}>1$ for every ${j}\in\{1,\ldots,q\}$. The reason to do this is that if $m_{j}=1$, then $g_{j}\in \Lor_{m_{j}}$ is a standard NLP constraint, which should be treated separately in our approach, together with equality constraints; we should remark that our approach is very friendly to this kind of mixed constraints, since it is based on~\cite{crcq-naive}. In particular, inclusion of equality constraints can be done in the way suggested in \cite{rcr} and \cite{ahss12}. Therefore, to avoid cumbersome notation, we will omit both types of NLP constraints in this paper.
Alternatively, the spectral decomposition of $y\in \Lor_1$ could be interpreted as $y=\lambda_1(y)u_1(y)$, with $u_1(y)=1$ and $\lambda_1(y)=y$. From this point of view, the definitions and theorems of this paper can be adjusted to fit the case $m_{j}=1$ by simply disregarding all expressions involving $\lambda_2(y)$ and $u_2(y)$. 
\end{remark}

Let $\xb\in \F$. The well-known \textit{Karush-Kuhn-Tucker} (KKT) conditions for $\xb$ consist of the existence of \textit{Lagrange multipliers} $\bar{\mu}_{j}\in \Lor_{m_{j}}$, ${j}\in\{1,\ldots,q\}$, such that
\begin{equation}
	\begin{aligned}
		\nabla_x L(\xb,\bar{\mu}_1,\ldots,\bar{\mu}_q) = 0, &\\
		\langle \bar{\mu}_{j}, g_{j}(\xb) \rangle = 0, & \quad \forall {j}\in \{1,\ldots,q\},
	\end{aligned}
\end{equation}
where
\[
	L(x,\mu_1,\ldots,\mu_q)\doteq f(x)-\sum_{{j}=1}^q \langle\mu_{j}, g_{j}(x) \rangle.
\]

It is known that not every local minimizer satisfies the KKT conditions, unless a constraint qualification is present. The most prominent constraint qualifications in the literature are the nondegeneracy CQ and Robinson's CQ, which we recall next as characterized\footnote{See~\cite[Proposition 19]{bonnansramirez} for the characterization of nondegeneracy. The characterization of Robinson's CQ follows from~\cite[Proposition 2.97 and Corollary 2.98]{bshapiro} using the fact $\langle y_{j}, g_{j}(\xb) \rangle=0$ with ${j}\in I_B(\xb)$ if, and only if, $y_{j}=\alpha \Gamma_{j} g_{j}(\xb)$ for some $\alpha\geq 0$; and similarly, $\langle y_{j}, g_{j}(\xb) \rangle=0$ with ${j}\in I_{\int}(\xb)$ if, and only if, $y_{j}=0$~\cite[Lemma 15]{surveysocp}.} in the work of Bonnans and Ram{\'i}rez~\cite{bonnansramirez}.
\begin{definition}\label{socp:ndgclass}
	A point $\xb\in \F$ satisfies 
	
	\begin{itemize}
		\item \emph{Nondegeneracy} if the family 
	\begin{equation}\label{socp:setndgclass}
		\left\{Dg_{j}(\xb)^\T \Gamma_{j} g_{j}(\xb)\right\}_{{j}\in I_B(\xb)}\bigcup \left\{Dg_{{j}}(\xb)^\T\right\}_{{j}\in I_0(\xb)}
	\end{equation}
	is $\R^{|I_B(\xb)|}\times \prod_{{j}\in I_0(\xb)}\R^{m_{j}}$-linearly independent;
	\item \emph{Robinson's CQ} if the family~\eqref{socp:setndgclass} is $\R^{|I_B(\xb)|}_+\times \prod_{{j}\in I_0(\xb)} \Lor_{m_{j}}$-linearly independent;
	\end{itemize}
	 where
	\begin{equation}\label{eq:gamma}
		\Gamma_{j}\doteq \begin{bmatrix}
		1 & 0 \\
		0 & -\I_{m_{j}-1}
		\end{bmatrix}
	\end{equation}
	and $\I_{m_{j}-1}$ is the identity matrix of dimension $m_{j}-1$.
\end{definition}

As mentioned in the introduction, the nondegeneracy condition reduces to LICQ from NLP when it is seen as an instance of~\eqref{NSOCP} with $m_1=\ldots=m_q=1$, while Robinson's CQ reduces to MFCQ in the same process.


\section{Weak Constraint Qualifications for NSOCP}
\label{sec:socp}

From the practical point of view, one of the standard strategies for proving first-order global convergence of iterative algorithms is proving that every feasible limit point $\xb$ of the sequence $\seq{x}$ of its iterates fulfills the KKT conditions whenever a given CQ holds. Roughly speaking, this means that the algorithm surely avoids all non-optimal points that satisfy the CQ but violate KKT; hence, building this reasoning under a more general (weaker) CQ means to narrow down the range of convergence of the method without removing optimal candidates from it. Moreover, it is well-known that the existence of Lagrange multipliers is a relevant issue beyond algorithms -- for example, in situations where they have some practical interpretation, such as in the electricity pricing context~\ima{(}\cite{lunasilvasaga}\ima{)} -- meaning there is also a theory-driven motivation for pursuing weaker constraint qualifications.

In this section, we will present weaker variants of nondegeneracy and Robinson's CQ, discuss some of their properties, and exemplify their usage with an external penalty method. Besides, these conditions shall pave the way for a more radical relaxation in terms of local constant rank, which will be discussed in the next section. A similar approach has been conducted in~\cite{seq-crcq-sdp,weaksparsecq} for NSDP problems, but although NSOCP can be seen as a particular case of NSDP via an arrowhead matrix transformation
\[
	(y_0,\widehat{y})\mapsto \textnormal{Arw}(y_0,\widehat{y})\doteq\left[\begin{array}{ccc|c}
	y_0 &  & & \\
	 & \ddots & & \widehat{y} \\
	 & & y_0 &\\
	 \midrule
	 & \widehat{y}^\T & & y_0
	\end{array}\right],
\] it should be noted that constraint qualifications are not necessarily carried over with the transformation; that is, when dealing with weak constraint qualifications, one generally loses information when the problem is equivalently rewritten differently (a noticeable exception is Robinson's CQ, which turns out the be quite robust in this sense). For instance, the nondegeneracy condition for NSDP is never satisfied by a constraint in the form
\[
	\textnormal{Arw}(g_0(x),\widehat{g}(x))\in \S^m_+\doteq \{M\in \R^{m\times m}\colon M=M^\T; \ \forall d\in \R^m, d^\T Md\geq 0\}
\]
when $m>2$, regardless of the fulfillment of nondegeneracy for NSOCP applied to the constraint \si{\linebreak}$(g_0(x),\widehat{g}(x))\in \Lor_m$. As it can be easily verified, the same conclusion holds for the constraint qualification called ``weak-nondegeneracy'' for NSDP that was introduced in~\cite{weaksparsecq}. Thus, a specialized analysis is required to obtain results similar to~\cite{seq-crcq-sdp,weaksparsecq}, for NSOCP. In fact, the analysis we present in this section regarding those weak conditions is, in a sense, more refined than the one presented in~\cite{weaksparsecq} since there are some important questions that were left open in~\cite{weaksparsecq}, which we are able to answer here.

\subsection{Parametric Bases and Weak-Nondegeneracy for NSOCP}

We open our studies by characterizing nondegeneracy and Robinson's CQ in terms of the eigenvectors of the constraint functions (as in~\eqref{socp:autovals}). To motivate it, let $g(x)\doteq (g_0(x),\widehat{g}(x))$ and $\xb\in \R^n$ be such that $g(\xb)=0$. Using Bonnans and Ramírez' characterization (Definition~\ref{socp:ndgclass}), we see that $\xb$ is \textit{nondegenerate} (that is, it satisfies nondegeneracy CQ) when the matrix $Dg(\xb)$ is surjective. This is clearly a representation of nondegeneracy in the canonical basis $e_1,\ldots,e_m$ of $\R^m$, where $e_i$ has $1$ in its $i$-th position and zeros elsewhere. Other representations of $\R^m$ may lead to different characterizations of these constraint qualifications; and this simple fact leads us a natural way of imbuing the structure of the cone into the conditions. 

For instance, the discussion after Lemma~\ref{lem:conicli} allows us to represent nondegeneracy and Robinson's CQ in terms of each slice of $\Lor_m$, as long as we consider all of them. More precisely: 
\begin{corollary}\label{socp:nondegen}
	Let $\xb$ be a feasible point of~\eqref{NSOCP}. Then:
	\begin{enumerate}
		\item Nondegeneracy holds at $\xb$ if, and only if, the family  of vectors
		\begin{equation}\label{socp:setndg}
			\left\{Dg_{j}(\xb)^\T u_1(g_{j}(\xb))\right\}_{{j}\in I_B(\xb)}\bigcup \left\{ Dg_{j}(\xb)^\T (1,-\bar{w}_{j}), \ Dg_{j}(\xb)^\T (1,\bar{w}_{j}) \right\}_{{j}\in I_0(\xb)}
		\end{equation}
		is linearly independent for every $\bar{w}_{j}\in \R^{m_j-1}$ such that $\norm{\bar{w}_{j}}=1$, ${j}\in I_0(\xb)$;
		
		\item Robinson's CQ holds at $\xb$ if, and only if, the family~\eqref{socp:setndg} is positively linearly independent for every $\bar{w}_{j}$ such that $\norm{\bar{w}_{j}}=1$, ${j}\in I_0(\xb)$.
	\end{enumerate}
\end{corollary}
\begin{proof}
	For item 2, it suffices to apply Lemma~\ref{lem:conicli} considering the product $C=\prod_{j\in J} C_j$, $J\doteq I_B(\xb)\cup I_0(\xb)$, where
	\[
		C_j\doteq\left\{
			\begin{array}{ll}
				\R_+, & \textnormal{if } j\in I_B(\xb),\\
				\Lor_{m_{j}}, & \textnormal{if } j\in I_0(\xb),
			\end{array}
		\right.
	\] 
	to the matrix $M=[M_j]_{j\in J}$ arranged as in~\eqref{eq:Mj}, whose blocks are given by
	\[
		M_j\doteq\left\{
			\begin{array}{ll}
				Dg_j(\xb)^\T u_1(g_j(\xb)), & \textnormal{if } j\in I_B(\xb),\\
				Dg_j(\xb)^\T, & \textnormal{if } j\in I_0(\xb).
			\end{array}
		\right.
	\]
	To see why $C$ fits the description of Lemma~\ref{lem:conicli}, define $S_j\doteq\{1\}$ for every $j\in I_B(\xb)$, $S_j\doteq\{\bar{w}_j\in \R^{m_j-1}\colon \|\bar{w}_j\|=1\}$ for every $j\in I_0(\xb)$; then, let $S\doteq\prod_{j\in J} S_j$ and for each $\bar{w}\doteq(\bar{w}_j)_{j\in J}\in S$, with $\bar{w}_j\in S_j$ for $j\in J$, define $\E_{\bar{w}}\doteq\prod_{j\in J} \E_{\bar{w}_j}$, where
	\[
		\E_{\bar{w}_j}\doteq \left\{
		\begin{array}{ll}
			1, & \textnormal{if } j\in I_B(\xb),\\
			\{(1,-\bar{w}_j), (1,\bar{w}_j)\}, & \textnormal{if } j\in I_0(\xb),
		\end{array}
		\right.
	\]
	for every $j\in J$. Observe that $C=\bigcup_{\bar{w}\in S}\textnormal{cone}\left(\E_{\bar{w}}\right)$ and the proof of item 2 is over. The proof for item 1 is similar, considering Remark \ref{remarklemma}.
\jo{\qed} \end{proof}
For a better understanding of the meaning of Corollary~\ref{socp:nondegen}, let us resume the short discussion after Lemma~\ref{lem:conicli}. Note that LICQ for a pair of constraints $g_1(x)\geq 0$ and $g_2(x)\geq 0$ at a point $\xb$ such that $g_1(\xb)=g_2(\xb)=0$, when seen through Corollary~\ref{socp:nondegen}, becomes equivalent to $Dg(\xb)^\T\left(\begin{array}{c}\cos(w)\\ \sin(w)\end{array}\right)$ being non-zero, for every $w\in [0,2\pi]$, where $g\doteq(g_1,g_2)$. On the one hand, this is obvious; but on the other hand, note that the process of checking linear independence of a couple of $n$-dimensional vectors is reduced to checking whether one $n$-dimensional vector is zero or not, for each fixed real parameter $w$. Of course, this reasoning can be extended to arbitrary dimensions and arbitrary parametrizations, and Corollary~\ref{socp:nondegen} is simply one of these extensions where the parametrization is given in terms of the second-order cone. This will turn out to be relevant in our analysis as we will be able to identify that some of the linear independence requirements will be superfluous for a constraint qualification to be defined. This kind of reasoning can also be applied to the cone of symmetric positive semidefinite matrices, leading to a different, in fact simpler, proof of~\cite[Proposition 3.2]{weaksparsecq}, which is the analogue of Corollary~\ref{socp:nondegen} in the context of NSDP, hence providing some intuition for a result that was originally presented as a mere technical tool in~\cite{weaksparsecq}.

With the characterization of Corollary~\ref{socp:nondegen} at hand, we can take a close look at a simple example that shall motivate our next steps:

\begin{example}\label{socp:motivation}
Let $g_{0}, g_1: \R^{n} \to \R$ be continuously differentiable functions, define $g\doteq (g_0,g_1)$, and let $\xb$ be a point such that:
\begin{itemize}
\item $g(\xb)=0$;
\item $\nabla g_0(\xb)$ and $\nabla g_1(\xb)$ are linearly independent.
\end{itemize}
Observe that nondegeneracy holds for the constraint $g(x)\in \Lor_2$ at $\xb$ since $Dg(\xb)^\T$ is $\R^2$-linearly independent. Now consider the equivalent NSOCP constraint 
\[
	\tilde{g}(x)\doteq(g_0(x),g_1(x),0, \ldots, 0) \in \Lor_m
\]
and observe that the KKT conditions for it are the same as for the constraint $g(x)\in \Lor_2$. However, by Corollary~\ref{socp:nondegen}, nondegeneracy for the reformulated problem is equivalent to the linear independence of the vectors
\[
	D\tilde{g}(\xb)^\T (1,-\bar{w})=\nabla g_0(\xb)-\bar{w}_1 \nabla g_1(\xb) \si{\mand}\ima{\mand} \jo{ \] and \[ } D\tilde{g}(\xb)^\T (1,\bar{w})=\nabla g_0(\xb)+\bar{w}_1 \nabla g_1(\xb)
\] for every $\bar{w}=(\bar{w}_1,\ldots,\bar{w}_{m-1})$ such that $\|\bar{w}\|=1$, which is violated when $\bar{w}_1=0$.

On the other hand, note that for every $x$ such that $g_1(x)\neq 0$ the eigenvectors of $\tilde{g}(x)$ are uniquely determined by
\[
	u_1(\tilde{g}(x))=\frac{1}{2}\left( 1, -\frac{g_1(x)}{|g_1(x)|},0 , \ldots, 0 \right) \si{\mand}\ima{\mand} \jo{ \] and \[ } u_2(\tilde{g}(x))=\frac{1}{2}\left( 1, \frac{g_1(x)}{|g_1(x)|},0 , \ldots, 0 \right).
\]
This suggests that although $\tilde g(\overline{x})$ admits multiple eigenvector decompositions $\frac{1}{2}(1,-\overline{w})$ and $\frac{1}{2}(1,\overline{w})$ with $\|\overline{w}\|=1$, the only relevant ones are \si{$\overline{w}=(\pm 1,0,\dots,0)$}\jo{$\overline{w}=(\pm 1,0,\dots, \linebreak 0)$}. That is, in light of our previous work in NSDP~\ima{(}\cite{weaksparsecq}\ima{)}, we can infer that the problematic choices of $\frac{1}{2}(1,-\bar{w})$ and $\frac{1}{2}(1,\bar{w})$ such that $\bar{w}_1=0$ may be disregarded when defining a constraint qualification. In fact, we may consider all sequences $\seq{x}\to \xb$ and we have that when $g_1(x^k)\neq 0$ for every $k\in \N$, the sequences $\{u_1(\tilde{g}(x^k))\}_{k\in \N}$ and $\{u_2(\tilde{g}(x^k))\}_{k\in \N}$ of eigenvectors of $\tilde{g}(x^k)$ are uniquely defined and $\frac{1}{2}(1,-\bar{w})$ and $\frac{1}{2}(1,\bar{w})$ with $\overline{w}_1=0$ are not among their limit points. Similarly, when $g_1(x^k)=0$ for some indexes $k\in\N$ one may also choose the eigendecompositions of $\tilde{g}(x^k)$ that avoids having $\frac{1}{2}(1,-\overline{w})$ and $\frac{1}{2}(1,\overline{w})$ with $\overline{w}_1=0$ as limit points.

Conversely, note that for any sequence $\seq{x}\to \xb$, the choice $\bar{w}=(\pm1,0,\ldots,0)$ does not present the same issue, and in this case we get that the vectors
\[
	D\tilde{g}(\xb)^\T (1,-\bar{w})=\nabla g_0(\xb)\mp\nabla g_1(\xb) \quad\textnormal{and}\quad D\tilde{g}(\xb)^\T (1,\bar{w})=\nabla g_0(\xb)\pm \nabla g_1(\xb)
\]
are linearly independent.
\end{example}

Example~\ref{socp:motivation} suggests that demanding linear independence of~\eqref{socp:setndg} for all $\bar{w}_{j}$ may be unnecessarily strong for a constraint qualification. In fact, it also suggests that only the limit points of sequences consisting of eigenvectors of $g(x^k)$, for each $\seq{x}\to \xb$, are needed. This observation leads to two new constraint qualifications for NSOCP:

\begin{definition}[Weak-nondegeneracy and weak-Robinson's CQ]\label{socp:wndgrobdef}
Let $\xb\in \F$. We say that $\xb$ satisfies:
\begin{itemize}
\item \emph{Weak-nondegeneracy} if, for each sequence $\seq{x}\to \xb$, there exists some $I\subseteq_{\infty} \N$ and convergent eigenvectors sequences \si{$\{u_1(g_{j}(x^k))\}_{k\in I}\to \frac{1}{2}(1,-\bar{w}_{j})$} \jo{$\{u_1(g_{j}(x^k))\}_{k\in I}\to \linebreak \frac{1}{2}(1,-\bar{w}_{j})$} and $\{u_2(g_{j}(x^k))\}_{k\in I}\to \frac{1}{2}(1,\bar{w}_{j})$, with $\|\bar{w}_{j}\| = 1$, for every ${j}\in I_0(\xb)$, such that~\eqref{socp:setndg} is linearly independent;

\item \emph{Weak-Robinson's CQ} if, for each sequence $\seq{x}\to \xb$, there exists some $I\subseteq_{\infty} \N$ and convergent eigenvectors sequences $\{u_1(g_{j}(x^k))\}_{k\in I}\to \frac{1}{2}(1,-\bar{w}_{j})$ and $\{u_2(g_{j}(x^k))\}_{k\in I}\to \frac{1}{2}(1,\bar{w}_{j})$, for every ${j}\in I_0(\xb)$, such that~\eqref{socp:setndg} is positively linearly independent;
\end{itemize}
where the notation $I\subseteq_{\infty}\N$ means that $I$ is an infinite subset of $\N$.
\end{definition}

Both conditions presented in Definition~\ref{socp:wndgrobdef} will be proved to be CQs later on; let us first discuss their properties and relations with other CQs. From Definition~\ref{socp:wndgrobdef}, it is clear that weak-nondegeneracy is implied by nondegeneracy, but the converse is not necessarily true, as illustrated by Example~\ref{socp:motivation}. Notice also that both conditions from Definition \ref{socp:wndgrobdef} are maintained under the addition of structural zeros as in Example~\ref{socp:motivation}, which somehow shows the robustness of the conditions we define. Similarly, for NSDPs, in \cite{weaksparsecq}, it is shown that the analogous conditions from Definition \ref{socp:wndgrobdef} are maintained when stacking several semidefinite constraints into a single block diagonal semidefinite constraint. The next example shows, however, that weak-nondegeneracy may hold when nondegeneracy fails even when the problem does not have structural zeros:

\begin{example}[Weak-nondegeneracy is weaker than Nondegeneracy]\label{socp:exwndgnotndg2}
Consider the constraint 
\[
	g(x) \doteq (x_1,x_2,x_2) \in \Lor_{3}
\]
at the point $\bar{x} \doteq (0,0)$, which does not satisfy nondegeneracy. Now, take any sequence $\seq{x} \rightarrow \xb$. There are three possible cases to consider:
\begin{enumerate}
\item There exists some infinite subset $I\subseteq_{\infty}\N$ such that $x_2^{k} > 0$ for all $k\in I$;
\item Case 1 fails to hold, but there exists some infinite subset $I\subseteq_{\infty}\N$ such that $x_2^{k} < 0$ for all $k\in I$;
\item Cases 1 and 2 both fail, implying $x_2=0$ for all $k$ large enough;
\end{enumerate}
In Case 1, the eigenvectors $u_{1}(g(x^k))$ and $u_{2}(g(x^k))$ are uniquely determined by
\[
u_{1}(g(x^k)) = \frac{1}{2}\left(1, -\frac{1}{\sqrt{2}}, -\frac{1}{\sqrt{2}} \right) \ \mand \ u_{2}(g(x^k)) = \frac{1}{2}\left(1, \frac{1}{\sqrt{2}}, \frac{1}{\sqrt{2}} \right),
\]
for all $k\in I$. Define $\bar{w} \doteq \left(\frac{1}{\sqrt{2}}, \frac{1}{\sqrt{2}} \right)$ and note that
\[
	\lim_{k\in I} u_{1}(g(x^k)) = \frac{1}{2}(1, -\bar{w}) \quad \textnormal{and} \quad \lim_{k\in I} u_{2}(g(x^k)) = \frac{1}{2}(1, \bar{w}).
\]
In addition, 
\[
	Dg(\xb)^{\T}(1, -\bar{w}) = \frac{1}{2} \left(\begin{array}{c} 1 \\ -\sqrt{2} \end{array}\right) \
	\mand \
	Dg(\xb)^{\T}(1, \bar{w}) = \frac{1}{2} \left(\begin{array}{c} 1 \\ \sqrt{2} \end{array}\right)
\]
are linearly independent. Case 2 is analogous. In Case 3, we have that the eigenvectors of $g(x^k)$ are not uniquely defined in \eqref{socp:autovals}; thus, in checking Definition \ref{socp:wndgrobdef} we may choose an appropriate eigendecomposition of each $g(x^k)$. In particular, we may pick the same decomposition analyzed previously to conclude that weak-nondegeneracy holds at $\xb$. Notice that since nondegeneracy fails, by Corollary \ref{socp:nondegen} there must exist some $\bar{w}, \|w\|=1$, such that $Dg(\xb)^\T(1,-\bar{w})$ and $Dg(\xb)^\T(1,\bar{w})$ are linearly dependent. This is the case of $\bar{w}\doteq\left(\frac{1}{\sqrt{2}},\frac{-1}{\sqrt{2}}\right)$ or $\bar{w}\doteq\left(\frac{-1}{\sqrt{2}},\frac{1}{\sqrt{2}}\right)$, however, since weak-nondegeneracy holds, these limit points can be avoided considering the eigendecompositions of $\{g(x^k)\}_{k\in \N}$ for any sequence $x^k\to\xb$.
\end{example}

At this point we acknowledge that weak-nondegeneracy may be hard to check. However, besides its robustness in terms of structural zeros as discussed in Example \ref{socp:motivation}, let us prove that there is a deeper connection between nondegeneracy and weak-nondegeneracy, in the sense that we may characterize nondegeneracy by the validity of weak-nondegeneracy plus a simple linear independence requirement of a partial family of derivative vectors in $I_0(\xb)$, namely, by removing from consideration in the family \eqref{socp:setndgclass} that defines nondegeneracy all gradients of first component entries, that is, $\nabla g_{j,0}(\xb), j\in I_0(\xb)$ together with the vectors indexed by $I_B(\xb)$. In fact, in Example \ref{socp:exwndgnotndg2}, this family of vectors reduces to the rows of $D\widehat{g}(\xb)$, where $\widehat{g}(x)\doteq(x_2,x_2)$, which are linearly dependent. Loosely speaking, weak-nondegeneracy may be thought as an appropriate form of nondegeneracy but without requiring linear independence of this partial family of vectors.

\begin{proposition}[Difference between weak-nondegeneracy and Nondegeneracy]\label{prop:ndgsurjective}
Let $\xb$ be a feasible point of (\ref{NSOCP}). We have that nondegeneracy holds at $\xb$ if, and only if, weak-nondegeneracy holds at $\xb$ and, in addition, the matrix
\begin{displaymath}
M \doteq \left[ \begin{array}{ccc}
	& \vdots &\\
     & D\widehat{g}_{{j}}(\bar{x}) & \\ 
     & \vdots &
\end{array} \right]_{{j} \in I_{0}(\xb)} 
\end{displaymath}
is surjective.
\end{proposition}

\begin{proof}
From Definition~\ref{socp:wndgrobdef} it is clear that if nondegeneracy holds at $\xb$, then weak-nondegeneracy also holds at $\xb$. Moreover, from~(\ref{socp:setndgclass}) we obtain that $M$ is surjective. Conversely, suppose that nondegeneracy does not hold at $\xb$. By Corollary~\ref{socp:nondegen}, there are unitary vectors $\bar{w}_{{j}} \in \R^{m_{{j}}-1}$, ${j} \in I_{0}(\xb)$, such that (\ref{socp:setndg}) is linearly dependent.

Let us define $\bar{w}=(\bar{w}_j)_{j\in I_0(\xb)}$. By the surjectivity of $M$, there exists a non-zero vector $d \in \R^{n}$ such that $\bar{w} = Md$. That is, we have that $D\widehat{g}_{{j}}(\xb)d=\bar{w}_j$ for all ${j} \in I_{0}(\xb)$. Now, take any positive sequence $\{t_{k}\}_{k\in \N} \to 0^{+}$ and let
\[
	x^{k} \doteq \xb + t_{k}d, \ \forall k\in \N.
\]
We have that $\seq{x} \to \xb$ and when we consider ${j} \in I_{0}(\xb)$ and the Taylor expansion of $\widehat{g}_{{j}}(x^{k})$ around $\xb$, we obtain that
\[
	\widehat{g}_{{j}}(x^{k}) = t_{k}\bar{w}_j + o(t_{k}) \neq 0
\] for all $k\in\N$ large enough, since $\bar{w}_j\neq0$. Moreover, for the indices ${j} \in I_{B}(\xb)$ we also have that $\widehat{g}_{{j}}(x^{k}) \neq 0$ for all $k$ large enough, because $\widehat{g}_j(\xb)\neq 0$. This means that the eigenvectors of $\widehat{g}_{{j}}(x^k)$ are uniquely determined from \eqref{socp:autovals} for all ${j} \in I_{0}(\xb) \cup I_{B}(\xb)$ and all $k\in \N$. In particular, for ${j} \in I_{0}(\xb)$ we have that

\begin{displaymath}
\dfrac{\widehat{g}_{{j}}(x^{k}) }{\|\widehat{g}_{{j}}(x^{k}) \|} = \dfrac{D\widehat{g}_{{j}}(\xb)d + o(t_{k})/t_{k}}{\|D\widehat{g}_{{j}}(\xb)d + o(t_{k})/t_{k} \|} \to 
\bar{w}_j.
\end{displaymath}
As a consequence, since $\bar{w}_{{j}} \in \R^{m_{{j}}-1}$, ${j} \in I_{0}(\xb)$, is such that \eqref{socp:setndg} is linearly dependent, we conclude that weak-nondegeneracy does not hold at $\xb$.
\jo{\qed} \end{proof}

The following example shows that although weak-nondegeneracy implies weak-Robinson's CQ, the converse is not true:

\begin{example}[Weak-Robinson is weaker than weak-nondegeneracy]\label{socp:weakrobnotweakndg}
Consider the constraint
\[
	g(x)\doteq (4x,2x,x)\in \Lor_3
\]
and the point $\xb\doteq 0$. Clearly, it satisfies Robinson's CQ, hence it also satisfies weak-Robinson's CQ.
However, observe that taking any sequence $\seq{x}\to \xb$ such that $x^k>0$ for all $k\in \N$, we have
\[
u_{1}(g(x^{k})) = \frac{1}{2}\left(1, -\frac{2}{\sqrt{5}}, -\frac{1}{\sqrt{5}}\right) \ \mand \ u_{2}(g(x^{k})) = \frac{1}{2}\left(1, \frac{2}{\sqrt{5}}, \frac{1}{\sqrt{5}}\right),
\] 
hence we have $u_{1}(g(x^{k}))\to\frac{1}{2}(1,-\bar{w})$ and $u_{2}(g(x^{k}))\to\frac{1}{2}(1,\bar{w})$ where $\bar{w}= \left(\frac{2}{\sqrt{5}}, \frac{1}{\sqrt{5}} \right)$. Then,
\[
	Dg(\xb)^{\T}(1,-\bar{w}) = \dfrac{4\sqrt{5} - 5}{2\sqrt{5}} > 0 \ \mand \ Dg(\xb)^{\T}(1,\bar{w}) = \dfrac{4\sqrt{5}+5}{2\sqrt{5}} > 0
\]
are linearly dependent, although positively linearly independent, implying that weak-nondegeneracy does not hold at $\xb$.
\end{example}

To discuss in detail the relation between weak-Robinson's CQ and Robinson's CQ for~(\ref{NSOCP}), we rely on a simple lemma:

\begin{lemma}\label{lemma:positively_li}
Let $\xb$ be a feasible point of (\ref{NSOCP}). If (weak-Robinson's CQ) weak-nondegeneracy holds at $\xb$, then the family of vectors
\begin{equation}\label{lemma:grad_vectors}
\left\{\nabla g_{{j},0}(\xb) \right\}_{{j} \in I_{0}(\xb)} \bigcup \left\{Dg_{{j}}(\xb)^{\T}u_{1}(g_{{j}}(\xb))\right\}_{{j} \in I_{B}(\xb)}
\end{equation}
is (positively) linearly independent. 
\end{lemma}

\begin{proof}
Assume that weak-Robinson's CQ holds at $\xb$, so there exists some vectors $\bar{w}_j\in \R^{m_j-1}$, $\|\bar{w}_j\|=1$, $j\in I_0(\xb)$, such that~\eqref{socp:setndg} is positively linearly independent; and, by contradiction, suppose that (\ref{lemma:grad_vectors}) is positively linearly dependent. Then, there are some $\eta_{{j}}\geq 0$, $j\in I_B(\xb)\cup I_0(\xb)$, not all zero, such that

\begin{equation}\label{aux:eq2}
\sum_{{j} \in  I_{0}(\bar{x})} \eta_{{j}} \nabla g_{{j},0}(\bar{x}) - \sum_{{j} \in I_{B}(\bar{x})} \eta_{{j}} Dg_{{j}}(\xb)^{\T}u_{1}(g_{{j}}(\xb))=0.
\end{equation}
Now set
\[
	\alpha_j=\beta_j=\frac{\eta_j}{2}
\]
for every $j\in I_0(\xb)$ and~\eqref{aux:eq2} can be rewritten as
\si{
\begin{equation*}
\sum_{{j}\in I_{0}(\bar{x})} \alpha_{{j}}Dg_{{j}}(\bar{x})^{\T}(1, -\bar{w}_{{j}}) + \sum_{{j}\in I_{0}(\bar{x})} \beta_{{j}}Dg(\bar{x})^{\T}(1, \bar{w}_{{j}}) + \sum_{{j} \in I_{B}(\bar{x})}\eta_{{j}}Dg_{{j}}(\bar{}x)^{\T}u_{1}(g_{{j}}(\bar{x}))=0, \\
\end{equation*}
}
\jo{
\begin{eqnarray*}
\sum_{{j}\in I_{0}(\bar{x})} \alpha_{{j}}Dg_{{j}}(\bar{x})^{\T}(1, -\bar{w}_{{j}}) + \sum_{{j}\in I_{0}(\bar{x})} \beta_{{j}}Dg(\bar{x})^{\T}(1, \bar{w}_{{j}}) + & & \\
\sum_{{j} \in I_{B}(\bar{x})}\eta_{{j}}Dg_{{j}}(\bar{}x)^{\T}u_{1}(g_{{j}}(\bar{x})) &=&0, 
\end{eqnarray*}
}
which implies~\eqref{socp:setndg} is positively linearly dependent, contradicting \si{weak-Robinson's CQ} \jo{weak-\linebreak Robinson's CQ}. The statement regarding weak-nondegeneracy follows analogously.
\jo{\qed} \end{proof}

Recall that Robinson's CQ can be evaluated separately for each of the constraints $g_{{j}}(x) \in \Lor_{m_{{j}}}$, $j\in \{1,\ldots,q\}$, and that this is weaker than Robinson's CQ when such system is regarded as a whole (however, not being a CQ). In fact, for any given $\xb\in \F$, the former can be characterized by the existence of some vectors $d_j\in \R^{n}$, $j\in \{1,\ldots,q\}$, such that $g_j(\xb)+Dg_j(\xb)d_j\in \int \Lor_{m_j}$, whereas the latter requires in addition $d_1=d_2=\ldots=d_q$ to hold. With this in mind, we prove next that weak-Robinson's CQ is somewhat in-between these two forms of Robinson's CQ.

\begin{theorem}\label{socp:wrob_rob_sep}
Consider Problem~(\ref{NSOCP}) and let $\xb \in \F$. If weak-Robinson's CQ holds at $\xb$, then for each index ${j}\in\{ 1, \ldots, q\}$ the point $\xb$ satisfies Robinson's CQ for the isolated constraint $g_{{j}}(x) \in \Lor_{m_{{j}}}$.
\end{theorem}
\begin{proof}
Let $\xb \in \F$ be a point such that weak-Robinson's CQ holds and assume that there exists an index $\ell \in \{1,\ldots,q\}$ such that Robinson's CQ does not hold. Then it follows by Lemma \ref{lemma:positively_li} that $g_{\ell} (\xb) = 0$. So there exists some $\bar{w}_{\ell}  \in \R^{m_{\ell} -1}$ such that $\|\bar{w}_{\ell} \|=1$ and the vectors $Dg_{\ell} (\xb)^\T(1,-\bar{w}_{\ell})$ and $Dg_{\ell}(\xb)^\T(1,\bar{w}_{\ell})$ are positively linearly dependent, that is, there exist scalars $\alpha\geq 0, \beta\geq 0$, at least one of them non-zero, such that
\[
	\alpha Dg_{\ell} (\xb)^{\T}(1, -\bar{w}_{\ell} ) + \beta Dg_{\ell} (\xb)^{\T}(1, \bar{w}_{\ell} )= 0.
\]
Defining $\tilde{w} \doteq \left( \frac{\beta - \alpha}{\alpha + \beta} \right) \bar{w}_{\ell} $, it follows that
\begin{equation}\label{socp:weak-R_equality}
\nabla g_{\ell, 0}(\xb)= - D\widehat{g}_{\ell} (\xb)^{\T}\tilde{w}.
\end{equation}
Note that $\|\tilde{w}\| \leq 1$, and that $\tilde{w} \not\in \Ker D\widehat{g}_{\ell}(\xb)^{\T}$; otherwise, $\nabla g_{\ell,0}(\xb) = 0$ and according to Lemma~\ref{lemma:positively_li} weak-Robinson's CQ fails.

Since $\Ker D\widehat{g}_{\ell}(\xb)^{\T}+\Im D\widehat{g}_{\ell}(\xb)=\R^{m_{\ell}-1}$, there exists some $v\in \Ker D\widehat{g}_{\ell}(\xb)^{\T}$ and some $d\in\R^{n}$ such that $\tilde{w}=v+D\widehat{g}_{\ell}(\xb)d$. Note that $D\widehat{g}_{\ell}(\xb)d\neq 0$, otherwise we would have that $\tilde{w} \in \Ker D\widehat{g}_{\ell}(\xb)^{\T}$. In addition, $0\neq\tilde{w} - v = \mathcal{P}_{\Im D\widehat{g}_{\ell}(\xb)}(\tilde{w})$ and by the non-expansiveness of the projection, we obtain $0<\|\tilde{w} - v\| \leq \|\tilde w\|\leq1$.

Now, proceeding similarly to the proof of Proposition~\ref{prop:ndgsurjective}, consider the sequence $\seq{x}$ given by $x^{k} \doteq \xb + t_{k}d$, for any positive scalars sequence $\{t_{k}\}_{k\in \N}\to 0^+$, and consider the Taylor expansion of $\widehat{g}_{\ell}(x^{k})$ around $\xb$:
\[
	\widehat{g}_{\ell}(x^{k})=t_{k} D\widehat{g}_{\ell}(\xb)d + o(t_{k}).
\]
Since $D\widehat{g}_{\ell}(\xb)d \neq 0$, it follows that there exists some $k_{0} \in \N$ such that $\widehat{g}_{\ell}(x^{k})\neq 0$ for every $k > k_{0}$, which implies that its eigenvectors, \si{$u_1(g_{\ell}(x^{k}))$ and $u_2(g_{\ell}(x^{k}))$} \jo{$u_1(g_{\ell}(x^{k}))$ and \linebreak $u_2(g_{\ell}(x^{k}))$}, are uniquely determined from \eqref{socp:autovals} for every $k > k_{0}$. Then we obtain that
\[
	\frac{\widehat{g}_{\ell}(x^k)}{\|\widehat{g}_{\ell}(x^k)\|}=\frac{D\widehat{g}_{\ell}(\xb)d+o(t_k)/t_k}{\|D\widehat{g}_{\ell}(\xb)d+o(t_k)/t_k\|} \to  \frac{\tilde{w}-v}{\|\tilde{w}-v\|}.
\]
It follows that
\[
	\lim_{k \rightarrow \infty} u_{1}(g_{\ell}(x^{k})) = \frac{1}{2}\left(1, - \frac{\tilde{w}-v}{\|\tilde{w}-v\|}\right) \si{\mand}\ima{\mand} \jo{ \] and \[ } \lim_{k \rightarrow \infty} u_{2}(g_{\ell}(x^{k})) = \frac{1}{2}\left(1, \frac{\tilde{w}-v}{\|\tilde{w}-v\|}\right)
\]
and, by weak-Robinson's CQ, the vectors \si{$Dg_{\ell}(\xb)^{\T}\left(1, - \frac{\tilde{w}-v}{\|\tilde{w}-v\|}\right)$ and $Dg_{\ell}(\xb)^{\T}\left(1, \frac{\tilde{w}-v}{\|\tilde{w}-v\|}\right)$} \jo{$Dg_{\ell}(\xb)^{\T}\left(1, - \frac{\tilde{w}-v}{\|\tilde{w}-v\|}\right)$ and \linebreak $Dg_{\ell}(\xb)^{\T}\left(1, \frac{\tilde{w}-v}{\|\tilde{w}-v\|}\right)$} are positively linearly independent. However, the following system in the variables $a$ and $b$:
\begin{eqnarray*}
0 &=& aDg_{\ell}(\xb)^{\T}\left(1, \frac{\tilde{w}-v}{\|\tilde{w}-v\|}\right) + b Dg_{\ell}(\xb)^{\T}\left(1, -\frac{\tilde{w}-v}{\|\tilde{w}-v\|}\right)  \\
 &=& a\nabla g_{\ell,0}(\xb) + \frac{a}{\|\tilde{w}-v\|}D\widehat{g}_{\ell}(\xb)^{\T}\tilde{w} + b\nabla g_{\ell,0}(\xb) - \frac{b}{\|\tilde{w}-v\|}D\widehat{g}_{\ell}(\xb)^{\T}\tilde{w}\\
 &=& \left[a\left(\frac{1}{\|\tilde{w}-v\|} - 1 \right) - b \left(\frac{1}{\|\tilde{w}-v\|}+1 \right) \right]D\widehat{g}_{\ell}(\xb)^{\T}\tilde{w} 
\end{eqnarray*}
has a nontrivial solution $a=1/\|\tilde{w}-v\| + 1 > 0$ and $b =1/\|\tilde{w}-v\| - 1 \geq 0$, which is a contradiction. In the second equality of the above chain, we used $D\widehat{g}_\ell(\xb)^\T v=0$; and in the last equality, we used~\eqref{socp:weak-R_equality}.
\jo{\qed} \end{proof}

\begin{remark}
The same strategy of the previous proof actually allows proving a slightly stronger result: if a feasible point $\xb$ satisfies weak-Robinson's CQ, then for each index ${j}\in I_0(\xb)$ the constraint
\[
	g_{\ell}(x) \in \Lor_{m_{\ell}}, \ \forall \ell\in I_B(\xb)\cup\{j\}
\] satisfies Robinson's CQ at $\xb$. In particular, if $I_0(\xb)$ is a singleton, then weak-Robinson's CQ and Robinson's CQ are equivalent, which is somewhat remarkable and highlights the ``robustness'' of Robinson's CQ. The situation where $I_0(\xb)$ is a singleton has been previously considered, for instance, in~\cite{mordukhovichauglag,outrataramirez}. In the general case we were not able to prove nor to provide a counterexample for the equivalence between Robinson's CQ and weak-Robinson's CQ. 
\end{remark}


\section{Constant Rank Conditions for NSOCP}\label{sec:cr}

Let us consider an NLP problem for a moment; that is,~\eqref{NSOCP} with $m_1=\ldots=m_q=1$, whose constraints take the form $g_1(x)\geq 0, \ldots, g_q(x)\geq 0$, and let $\xb\in \F$. We recall that the nondegeneracy condition in this case is equivalent to LICQ, which holds when the family of vectors
\begin{equation}\label{eq:licq}
	\left\{ \nabla g_{j}(\xb)\right\}_{{j}\in I_0(\xb)}
\end{equation}
has full rank. The constant rank constraint qualification (CRCQ) condition can be considered a relaxation of LICQ, since it allows the rank of~\eqref{eq:licq} to be incomplete, as long as the rank of the family
\begin{equation}\label{eq:crcq}
	\left\{ \nabla g_{j}(x)\right\}_{{j}\in J_0}
\end{equation}
remains constant in a neighborhood of $\xb$, for every subset $J_0\subseteq I_0(\xb)$. \si{Qi and Wei}\jo{Qi and Wei}~\cite{qiwei} described CRCQ in a slightly different but equivalent way: CRCQ holds at $\xb$ if, for every $J_0\subseteq I_0(\xb)$, if \eqref{eq:crcq} is linearly dependent at $\xb$, then it must also remain linearly dependent for every $x$ in a neighborhood of $\xb$. Similarly, Robinson's CQ is equivalent to the positive linear independence of~\eqref{eq:licq}, and the relaxation of it in the same style as CRCQ characterizes the constraint qualification known as constant positive linear dependence (CPLD)\ima{ -- see}~\cite{Andreani2005}. That is, CPLD holds at $\xb$ if, for every subset $J_0\subseteq I_0(\xb)$, if \eqref{eq:crcq} is positively linearly dependent at $\xb$, then it must remain linearly dependent for every $x$ in a neighborhood of $\xb$.

Extending such constant rank-type constraint qualifications to the context of NSOCP with an arbitrary dimension is not trivial. For instance, it is known that linear second-order cone programming problems may present a positive or infinite duality gap even when the primal problem is bounded, feasible and its solution is attained. This means that ``constraint linearity'' is not a constraint qualification in NSOCP, contrary to NLP. However, note that any kind of constant rank condition that depends solely on the derivatives of the constraint functions will always be satisfied for every linear problem, implying it cannot be a constraint qualification -- see, for instance, ~\cite{ZZerratum}. See also~\cite[Section 2.1]{facial-crcq} for a detailed discussion on this issue regarding linear problems.

In a previous work~\ima{(}\cite{seq-crcq-sdp}\ima{)} we noticed that weak-nondegeneracy imbues the cone structure into the constraint functions, allowing us to properly define a constant rank-type condition that is not retained by the linearity bottleneck. In this section, we shall follow a similar approach, making the necessary adaptations to overcome the difficulties that arise from the particularities of the second-order cone along the way.

\subsection{Weak Constant Rank Conditions}

With\jo{ the definitions of} weak-nondegeneracy and weak-Robinson's CQ for NSOCP at hand, we can present new extensions of \nlpcrcq{} and \nlpcpld{} for NSOCP by means of a simple relaxation of Definition~\ref{socp:wndgrobdef}, in the same lines as in NLP. Basically, the idea is to demand every subfamily of~\eqref{socp:setndg} to locally retain its (positive) linear dependence. So let us define, for any sets $J_B,J_{-},J_{+}\subseteq\{1,\dots,q\}$ such that $\widehat{g}_j(x)\neq 0$ for every $j\in J_B$, the family of vectors
\si{
\begin{equation}\label{socp:familyd}
	\mathcal{D}_{J_B,J_{-},J_{+}}\left(x,w\right)\doteq \left\{Dg_{j}(x)^\T u_1(g_j(x))\right\}_{{j}\in J_B}\bigcup \left\{Dg_{j}(x)^\T (1,-w_j)\right\}_{{j}\in J_-}\bigcup \left\{Dg_{j}(x)^\T (1,w_j)\right\}_{{j}\in J_{+}}
\end{equation}
}
\ima{
\begin{eqnarray}\label{socp:familyd}
\mathcal{D}_{J_B,J_{-},J_{+}}\left(x,w\right) &\doteq& \left\{Dg_{j}(x)^\T u_1(g_j(x))\right\}_{{j}\in J_B}\bigcup \left\{Dg_{j}(x)^\T (1,-w_j)\right\}_{{j}\in J_-} \nonumber \\ 
 & & \bigcup \left\{Dg_{j}(x)^\T (1,w_j)\right\}_{{j}\in J_{+}}
\end{eqnarray}
}
\jo{
\begin{eqnarray}\label{socp:familyd}
\mathcal{D}_{J_B,J_{-},J_{+}}\left(x,w\right) &\doteq& \left\{Dg_{j}(x)^\T u_1(g_j(x))\right\}_{{j}\in J_B}\bigcup \left\{Dg_{j}(x)^\T (1,-w_j)\right\}_{{j}\in J_-} \nonumber \\ 
 & & \bigcup \left\{Dg_{j}(x)^\T (1,w_j)\right\}_{{j}\in J_{+}}
\end{eqnarray}
}
where $w=[w_j]_{j\in J_-\cup J_+}$. Above, the index set $J_B$ refers to an arbitrary subset of $I_B(\xb)$, and the indices $J_{-}$ and $J_{+}$ both refer to $I_0(\xb)$, but with distinct eigenvectors; see~\eqref{socp:setndg}.

\begin{definition}[\socpwcrcq{} and \socpwcpld{}]\label{socp:crcq}
	We say that a feasible point $\xb$ of~\eqref{NSOCP} satisfies the:
	\begin{itemize}
		\item \emph{Weak constant rank constraint qualification} (\socpwcrcq{}) if the following holds: for every sequence $\seq{x}\to \xb$, there exists some $I\subseteq_{\infty} \N$, and convergent eigenvector sequences
\[
	\{u_1(g_{j}(x^k))\}_{k\in I}\to \frac{1}{2}(1,-\bar{w}_{j}) \ \mand \ \{u_2(g_{j}(x^k))\}_{k\in I}\to \frac{1}{2}(1,\bar{w}_{j}),\] with $\|\bar{w}_{j}\| = 1$, for all ${j}\in I_0(\xb)$, such that for all subsets $J_B\subseteq I_B(\xb)$ and $J_{-},J_{+}\subseteq I_0(\xb)$, we have that: if the family of vectors $\mathcal{D}_{J_B,J_{-},J_{+}}(\xb,\bar{w})$ is linearly dependent, then $\mathcal{D}_{J_B,J_{-},J_{+}}(x^k,w^k)$ remains linearly dependent for all $k\in I$ large enough, where $\bar{w}=[\bar{w}_j]_{j\in J_-\cup J_+}$ and $w^k=[w^k_j]_{j\in J_-\cup J_+}$ satisfies
\begin{equation}\label{socp:setJ0}
u_1(g_j(x^k))=\frac{1}{2}(1,-w^k_j) \quad\textnormal{and}\quad u_2(g_j(x^k))=\frac{1}{2}(1,w^k_j)
\end{equation}
for each $j\in J_-\cup J_+$.
	\item \emph{Weak constant positive linear dependence} (\socpwcpld{}) condition if the following holds: for every sequence $\seq{x}\to \xb$, there is some $I\subseteq_{\infty} \N$, and convergent eigenvector sequences \[
	\{u_1(g_{j}(x^k))\}_{k\in I}\to \frac{1}{2}(1,-\bar{w}_{j}) \ \mand \ \{u_2(g_{j}(x^k))\}_{k\in I}\to \frac{1}{2}(1,\bar{w}_{j}),\]
with $\|\bar{w}_{j}\| = 1$, for all ${j}\in I_0(\xb)$, such that for all subsets $J_B\subseteq I_B(\xb)$ and $J_{-},J_{+}\subseteq I_0(\xb)$, we have that: if \ima{\linebreak} $\mathcal{D}_{J_B,J_{-},J_{+}}(\xb,\bar{w})$ is positively linearly dependent, then $\mathcal{D}_{J_B,J_{-},J_{+}}(x^k,w^k)$ is linearly dependent for all $k\in I$ large enough, where $\bar{w}$ and $w^k$ are as in the previous item.
	\end{itemize}			
	
\end{definition}
	There are some features about Definition~\ref{socp:crcq} that should be highlighted for a better understanding of it. First, \socpwcrcq{} fully recovers \nlpcrcq{} when we set $m_{j}=1$ for every ${j}\in \{1,\ldots,q\}$ -- see also Remark~\ref{socp:remark0} for a clarification about the case $m_{j}=1$. Similarly, note that \socpwcpld{} recovers \nlpcpld{} in the same setting. Second, in view of Corollary~\ref{socp:nondegen}, we see that \socpwcrcq{} is implied by (weak-)nondegeneracy as in Definition~\ref{socp:wndgrobdef}, and \socpwcpld{} is implied by both (weak-)Robinson's CQ and \socpwcrcq{}. However, due to such equivalence in NLP, those implications in the conic setting are strict (see Example~\ref{socp:crnotweakcq} below and~\cite[Counterexample 4.2]{Andreani2005}, respectively). Third, we point out that \socpwcrcq{} is not comparable with (weak-)Robinson's CQ (see, for instance,~\cite[Examples 2.1 and 2.2]{janin}).
\begin{remark}
To fix ideas, let us consider a single conic constraint $g(x)\in \Lor_m$ at the point $\xb\in \F$. First, suppose that $g(\xb)=0$ and take any sequence $\seq{x}\to \xb$. We consider a partition of $\mathbb{N}$ as follows:
\begin{itemize}
\item $\mathcal{N}_0\doteq \{k\in \N\colon \widehat{g}(x^k)=0\}$. For $k\in\mathcal{N}_0$, we can choose
\[
	u_1(g(x^k))=\frac{1}{2}\left(1,-w^k\right)\quad \mand \quad u_2(g(x^k))=\frac{1}{2}\left(1,w^k\right),\] for any $w^k$ such that $\|w^k\|=1$. When $\mathcal{N}_0$ is infinite, \socpwcrcq{} demands, in particular, the existence of a choice of $\{w_k\}_{k\in\mathcal{N}_0}$ with some convergent subsequence $\{w^k\}_{k\in I}\to \bar{w}$, $I\subseteq_\infty\mathcal{N}_0$, such that \jo{\begin{center}}$Dg(\xb)^\T(1,(-1)^i\bar{w})=0$\jo{\end{center}} only if \jo{\begin{center}}$Dg(x^k)^\T \left(1,(-1)^i w^k\right)=0$\jo{\end{center}} for all large $k\in I$, $i\in \{1,2\}$; and, in addition, if $Dg(\xb)^\T(1,-\bar{w})$ and $Dg(\xb)^\T(1,\bar{w})$ are linearly dependent, then \si{$Dg(x^k)^\T \left(1,-w^k\right)$ and $Dg(x^k)^\T \left(1,w^k\right)$} \jo{$Dg(x^k)^\T \left(1,-w^k\right)$ and \linebreak $Dg(x^k)^\T \left(1,w^k\right)$} must also be linearly dependent, for every sufficiently large $k\in I$.

\item $\mathcal{N}_1\doteq \{k\in \N\colon \widehat{g}(x^k)\neq 0\}$. This case is similar to the previous one, except that there is no freedom in the choice of $w^k$, as it  is uniquely determined by $w^k={\widehat{g}(x^k)}/{\|\widehat{g}(x^k)\|}$, for every $k\in \mathcal{N}_1$.
\end{itemize}

The reason why both eigenvectors are taken into consideration is that both eigenvalues of $g(\xb)$ are zero, in this case. Naturally, in case $g(\xb)\in \bd_+\Lor_m$, we have only one zero eigenvalue, which is $\lambda_1(g(\xb))$, then \socpwcrcq{} simply demands the vector 
\[
	Dg(x)^\T u_1(g(x))=\frac{1}{2}Dg(x)^\T\left( 1, -\frac{\widehat{g}(x)}{\|\widehat{g}(x)\|} \right)
\]
to be either non-zero at $\xb$ or equal to zero in a whole neighborhood of $\xb$. Note that this coincides with the naive approach \ima{of~}\cite{crcq-naive}, obtained by reducing the problem to an NLP. This observation remains true for more than one conic constraint as long as $I_0(\xb)=\emptyset$. See also Remark~\ref{naive-remark} below.
\end{remark}

Now, let us check how Definition~\ref{socp:crcq} behaves when it is applied to \ima{the} example\ima{ of}~\cite[Equation 2]{ZZerratum}, which was used to refute the CRCQ proposal of~\cite{ZZ}.

\begin{example}[Equation 2 from~\cite{ZZerratum}]
Consider the problem
\begin{equation}\label{exZZerratum}
  \begin{aligned}
    & \underset{x \in \mathbb{R}}{\textnormal{Minimize}}
    & & -x, \\
    & \textnormal{subject to}
    & & g(x)\doteq (x,x+x^2)\in \Lor_2.\\
  \end{aligned}
\end{equation}
and its unique feasible point $\xb\doteq 0$, which does not satisfy the KKT conditions. Our aim is to show that Definition~\ref{socp:crcq} is not satisfied at $\xb$. To do so, it suffices to take any sequence $\seq{x}\to 0$ such that $x^k>0$ for all $k\in \N$. In this case, for each $k\in \N$, the eigenvectors of $g(x^k)$ are uniquely determined by
\begin{equation*}
		u_1(g(x^k))=\frac{1}{2}\left(1,-\frac{x^k+(x^k)^2}{|x^k+(x^k)^2|}\right)=\frac{1}{2}(1,-1) 
\end{equation*}
and
\begin{equation*}
		u_2(g(x^k))=\frac{1}{2}\left(1,\frac{x^k+(x^k)^2}{|x^k+(x^k)^2|}\right)=\frac{1}{2}(1,1),
\end{equation*}
so there is only one trivial limit point for each eigenvector sequence; also, $w^k=\bar{w}=1$ for every $k\in\N$. However, note that 
\[
	Dg(\overline{x})^\T(1,-\bar{w}) = 0
	\quad\textnormal{but}\quad
	Dg(x^k)^\T (1,-w^k) = -2x^k,
\]
so for $J_B\doteq I_B(\xb)=\emptyset$, $J_{-}\doteq\{1\}$, and $J_{+} \doteq \emptyset$, we have $\mathcal{D}_{J_B,J_-,J_+}(x^k,w^k)=\{-2x^k\}$ is linearly independent for every $k\in \N$ whereas $\mathcal{D}_{J_B,J_-,J_+}(\xb,\bar{w})=\{0\}$ is (positively) linearly dependent. Thus, neither \socpwcrcq{} nor \socpwcpld{} are satisfied at $\xb$.
\end{example}

As mentioned before, weak-nondegeneracy and weak-Robinson's CQ are strictly stronger than \socpwcrcq{} and \socpwcpld{}, respectively. It is clear that the former implies the latter, so let us prove the ``strict'' statement:
\begin{example}[Weak-CRCQ is weaker than weak-nondegeneracy and does not imply weak-Robinson]\label{socp:crnotweakcq}
Consider the constraint 
\[
	g(x)\doteq \left(-x,x,x\right)\in \Lor_3,
\]
and its unique feasible point $\xb \doteq 0$. To prove that~\socpwcpld{} holds at $\xb$, let $\seq{x}\to \xb$ be any sequence. Just as in Example~\ref{socp:exwndgnotndg2}, there are three cases to be considered but it suffices to analyse one of them, since the other cases follow analogously. Then, for simplicity, we assume that there is some $I\subseteq_{\infty} \N$ such that $x^k>0$ for every $k\in I$, and in this case the eigenvectors of $g(x^k)$ are uniquely determined by 
\[
	u_1(g(x^k))=\frac{1}{2}\left( 1 , -\frac{1}{\sqrt{2}}, -\frac{1}{\sqrt{2}} \right) \ \mand \ u_2(g(x^k))=\frac{1}{2}\left( 1 , \frac{1}{\sqrt{2}}, \frac{1}{\sqrt{2}} \right),
\]
leading to $w^k=\bar{w}=\left(\frac{1}{\sqrt{2}},\frac{1}{\sqrt{2}}\right)$. Then,
\[
	Dg(x^k)^\T (1,(-1)^i w^k) = Dg(\xb)^\T (1,-\bar{w}) = \left(-1 + (-1)^i \frac{2}{\sqrt{2}} \right) < 0
\]
for each $i\in \{1,2\}$. Then, the family~\eqref{socp:setndg} will have the same sign, making it (positively) linearly dependent, so weak-Robinson's CQ and \si{weak-nondegeneracy} \jo{weak-\linebreak nondegeneracy} both fail at $\xb$, without violating the weak-CRCQ and weak-CPLD requirements since in this example \jo{ \[ }\ima{ $ }\si{ $ }\mathcal{D}_{J_B,J_-,J_+}(x^k,w^k)=\mathcal{D}_{J_B,J_-,J_+}(\xb,\bar{w})\si{ $ }\ima{ $ }\jo{ \] } for every $k\in I$ regardless of $J_B, J_-$, and $J_+$.
\end{example}

Example~\ref{socp:crnotweakcq} can also be used to verify that \socpwcrcq{} does not imply Robinson's CQ. In fact, Robinson's CQ does not imply \socpwcrcq{} either, making them independent. Let us show this with another example:

\begin{example}[Weak-Robinson does not imply weak-CRCQ]\label{socp:exrobcrcq}
Consider the constraint
\[
	g(x)\doteq (2x_{1},x_{2}^{2})\in \Lor_2
\]
at $\xb\doteq 0.$ To see that $\xb$ violates~\socpwcrcq{}, it is enough to take any sequence $\seq{x}\to \xb$ such that $x^k\neq 0$ for every $k\in\N$. Then, the eigenvectors of $g(x^k)$ must be
\[
	u_1(g(x^k))=\frac{1}{2}(1,-1) \quad \mand \quad u_2(g(x^k))=\frac{1}{2}(1,1),
\]
which are defined by $w^k=\bar{w}=1$ for all $k\in \N$. This implies that the vectors $Dg(x^k)^\T (1,-w^k)=(1,-2x^k_2)$ and $Dg(x^k)^\T (1,w^k)=(1,2x^k_2)$ are linearly independent for all $k$, whereas the vectors \ima{\linebreak} $Dg(\xb)^\T (1,-\bar{w})=(1,0)$ and $Dg(\xb)^\T (1,\bar{w})=(1,0)$ are linearly dependent, violating~\socpwcrcq{}.

On the other hand, in view of Corollary~\ref{socp:nondegen}, it is easy to check that Robinson's CQ holds at $\xb$, since $Dg(\xb)^\T (1,-\bar{w})=(1,0)$ and $Dg(\xb)^\T (1,\bar{w})=(1,0)$ are positively linearly independent for every $\bar{w}\in\R$ with $|\bar{w}|=1$.
\end{example}

Finally, we shall prove that \socpwcpld{} (and by consequence \socpwcrcq{}, weak-nondegeneracy, and weak-Robinson's CQ) is a constraint qualification for~\eqref{NSOCP} employing a result from~\cite{Santos2019}, regarding the output sequences of an external penalty method:

\begin{theorem}\label{thm:minakkt}
	Let $\xb$ be a local minimizer of~\eqref{NSOCP}, and let $\{\rho_k\}_{k\in \N}\to +\infty$. Then, there exists some sequence $\seq{x}\to \xb$, such that for each $k\in \N$, $x^k$ is a local minimizer of the regularized penalized function 
	\begin{equation}\label{penfunc}
		f(x)+\frac{1}{2}\enorm{x-\xb}^2 
	  	+ \frac{\rho_k}{2} \left(\sum_{{j}=1}^q \norm{\mathcal{P}_{\Lor_{m_{j}}}(-g_{j}(x))}^2\right).
	\end{equation}
\end{theorem}
\begin{proof}
The proof of this theorem is contained in the proof of~\cite[Theorem 3.1]{Santos2019}.
\jo{\qed} \end{proof}

Observe that the gradient of~\eqref{penfunc} can be computed as
\[
	\nabla_x L\left(x,\rho_k\mathcal{P}_{\Lor_{m_1}}(-g_1(x)),\ldots,\rho_k\mathcal{P}_{\Lor_{m_q}}(-g_q(x))\right)+(x-\xb),
\]
for each $k\in \N$, which vanish at $x\doteq x^k$. So defining $\mu^k_{j}\doteq \rho_k\mathcal{P}_{\Lor_{m_{j}}}(-g_{j}(x^k))$, for all ${j}\in \{1,\ldots,q\}$, induces approximate Lagrange multiplier sequences associated with $\seq{x}$ -- see also~\cite{Santos2019}. Then, to prove that \socpwcpld{} is a CQ, it suffices to construct bounded approximate multiplier sequences out of $\seq{\mu_{j}}$. For convenience, we will prove a slightly more general result that also encompasses the convergence theory of an external penalty method under \socpwcpld{}; see~\cite{Santos2019} for details.
\begin{theorem}[Weak-Robinson, weak-CRCQ and weak-CPLD are constraint qualifications]\label{socp:mincrcqkkt}
	Let $\{\rho_k\}_{k\in \N}\to \infty$ and $\seq{x}\to\xb\in\F$ be such that
	\[
		\nabla_x L\left(x^k,\rho_k\mathcal{P}_{\Lor_{m_1}}(-g_1(x^k)),\ldots,\rho_k\mathcal{P}_{\Lor_{m_q}}(-g_q(x^k))\right)\to 0,
	\] 
	and suppose that \socpwcpld{} holds at $\xb$. Then,  $\xb$ satisfies the KKT conditions.
	Moreover, any local minimizer of~\eqref{NSOCP} that satisfies \socpwcpld{} is a KKT point. 
\end{theorem}
\begin{proof}
	For each $k\in \N$ and ${j}\in\{1,\ldots,q\}$, define $\mu^k_{j}\doteq \rho_k\mathcal{P}_{\Lor_{m_{j}}}(-g_{j}(x^k))$. Then, we have
	\begin{equation}\label{socp:approxstat}
		\nabla f(x^k)-\sum_{{j}=1}^q Dg_{j}(x^k)^\T \mu^k_{j}\to 0.
	\end{equation}
	Let us consider an arbitrary spectral decomposition of $\mu^k_{j}$:
	\[	
		\mu^k_{j}=\alpha_{j}^ku_1(g_{j}(x^k))+\beta_{j}^ku_2(g_{j}(x^k)),
	\]
	where $\alpha_{j}^k=[-\rho_k\lambda_1(g_{j}(x^k))]_+\geq 0$ and $\beta_{j}^k=[-\rho_k\lambda_2(g_{j}(x^k))]_+\geq0$. Define
	\begin{equation}\label{socp:decompstat}
		\begin{aligned}
		\Psi^k \doteq \jo{&} \sum_{{j}\in I_B(\xb)\cup I_0(\xb)} \alpha_{j}^k Dg_{j}(x^k)^\T u_1(g_{j}(x^k))
		+ \jo{ \\ & +}
		\sum_{{j}\in I_0(\xb)} \beta_{j}^k Dg_{j}(x^k)^\T u_2(g_{j}(x^k))
		\end{aligned}
	\end{equation}
	and note that~\eqref{socp:approxstat} can be equivalently stated as $\nabla f(x^k)-\Psi^k\to 0$. 
	%
	By Carath{\'e}odory's Lemma~\ref{lem:carath}, for each $k\in \N$, there exists some $J^k_B\subseteq I_B(\xb)$ and $J^k_-,J^k_+\subseteq I_0(\xb)$ such that
\begin{equation}\label{eq:setDmidproof}
		\left\{ Dg_{j}(x^k)^\T u_1(g_{j}(x^k))\right\}_{{j}\in J_B^k\cup J_{-}^k}
		\bigcup 
		\left\{Dg_{j}(x^k)^\T u_2(g_{j}(x^k)) \right\}_{{j}\in J^k_+}
	\end{equation}
	is linearly independent and 
	\[
		\Psi^k= \sum_{{j}\in J^k_B\cup J_{-}^k} \tilde{\alpha}_{j}^k Dg_{j}(x^k)^\T u_1(g_{j}(x^k))+
		\sum_{{j}\in  J^k_+} 
		\tilde{\beta}_{j}^k Dg_{j}(x^k)^\T u_2(g_{j}(x^k)),
	\]
	for some new scalars $\tilde{\alpha}_{j}^k\geq0$, ${j}\in J_B^k\cup J_{-}^k$, and $\tilde{\beta}_{j}^k\geq0$, $j\in J_{+}^k$. By the infinite pigeonhole principle, we can take a subsequence if necessary such that $J_B^k$, $J_{-}^k$, and $J_{+}^k$ do not depend on $k$; that is, we can assume without loss of generality that $J_B^k=J_B$, $J_-^k=J_{-}$, and $J_+^k= J_{+}$, for every $k\in \N$.
	
	We claim that the sequences $\seq{\tilde{\alpha}_{j}}$ are bounded for every ${j}\in J_B\cup J_{-}$, as well as $\seq{\tilde{\beta}_{j}}$ for every $j\in J_{+}$. Indeed, by contradiction, suppose that the sequence $\seq{m}$, given by \[
	m^k\doteq \max\{\max\{\tilde\alpha_{j}^k\colon {j}\in J_B\cup J_{-}\}, \ \max\{\tilde\beta_{j}^k\colon {j}\in J_{+}\}\},
\] diverges. Dividing~\eqref{socp:approxstat} by $m^k$, we obtain
	\[
		\sum_{{j}\in J_B\cup J_{-}} \frac{\tilde{\alpha}_{j}^k}{m^k} Dg_{j}(x^k)^\T u_1(g_{j}(x^k))+
		\sum_{{j}\in J_{+}} 
		\frac{\tilde{\beta}_{j}^k}{m^k} Dg_{j}(x^k)^\T u_2(g_{j}(x^k))
		\to 0
	\]
	and since the sequences $\{\tilde{\alpha}_{j}^k/m^k\}_{k\in \N}$ are bounded, we can assume without loss of generality, that they converge to, say, $\bar{\alpha}_{j}\geq 0$, for all ${j}\in J_B\cup J_{-}$; and, similarly, we can also assume that the sequences $\{\tilde{\beta}_{j}^k/m^k\}_{k\in \N}$ converge to some $\bar{\beta}_{{j}}\geq0$, for all ${j}\in J_{+}$. Note that at least one element of $\{\bar{\alpha}_{j}\}_{j\in J_B\cup J_-}\cup \{\bar{\beta}_{{j}}\}_{{j}\in  J_{+}}$ is non-zero, which makes the correspondent set $\mathcal{D}_{J_B,J_-,J_+}(\xb,\bar{w})$ as in Definition~\ref{socp:crcq} linearly dependent for any limit point $\bar{w}$ of any subsequence of $\seq{w}$, contradicting \socpwcpld{} since $\mathcal{D}_{J_B,J_-,J_+}(x^k,w^k)$, which coincides with~\eqref{eq:setDmidproof} with $w^k$ defined as in~\eqref{socp:setJ0}, is linearly independent for every $k\in \N$.
	
	Since $\seq{\tilde{\alpha}_{j}}$ and $\seq{\tilde{\beta}_{j}}$ are bounded, the sequence \si{$\{(\tilde{\mu}_1^k, \ldots, \tilde{\mu}_q^k)\}_{k\in \N}\subseteq \Lor_{m_1}\times\dots\times\Lor_{m_q}$} \jo{$\{(\tilde{\mu}_1^k, \ldots, \tilde{\mu}_q^k)\}_{k\in \N} \linebreak \subseteq \Lor_{m_1}\times\dots\times\Lor_{m_q}$} defined by
	\[
		\tilde{\mu}_{j}^k\doteq\left\{
		\begin{array}{ll}
			\tilde{\alpha}_{{j}}^k u_1(g_{j}(x^k)) + \tilde{\beta}_{{j}}^k u_2(g_{j}(x^k)), & \textnormal{ if } {j}\in J_{-}\cap J_{+},\\
			\tilde{\alpha}_{{j}}^k u_1(g_{j}(x^k)), & \textnormal{ if } {j}\in J_B\cup (J_{-}\setminus J_{+}),\\
			\tilde{\alpha}_{{j}}^k u_2(g_{j}(x^k)), & \textnormal{ if } {j}\in J_{+}\setminus J_{-},\\
			0, & \textnormal{ if } {j}\in I_{\int}(\xb) \textnormal{ or } {j}\not\in (J_B\cup J_{-}\cup J_{+})
		\end{array}
		\right.
	\]
	is also bounded. Finally, note that all limit points of $\{(\tilde{\mu}_1^k, \ldots, \tilde{\mu}_q^k)\}_{k\in \N}$ are Lagrange multipliers associated with $\xb$, which completes the first part of the proof. The second part follows directly from Theorem~\ref{thm:minakkt}.
\jo{\qed} \end{proof}

\begin{remark}\label{naive-remark}
In~\cite[Section 5]{crcq-naive}, we proposed so-called ``naive extensions'' of \nlpcrcq{} (and \nlpcpld{}) to NSOCP, which were obtained by replacing the conic constraints of~\eqref{NSOCP} that satisfy $g_{j}(\xb)\in \bd_+{\Lor_{m_{j}}}$ with standard NLP constraints, via a reduction function 
\[
	\Phi_{j}(x)\doteq g_{j,0}(x)^2 - \|\widehat{g}_{j}(x)\|^2,
\]
and then applying the NLP definition of \nlpcrcq{} (respectively, \nlpcpld{}) to those reduced constraints. However, in order to compare it with the conditions we presented, we use another reduction function, 
\[
	\tilde\Phi_{j}(x)\doteq g_{j,0}(x) - \|\widehat{g}_{j}(x)\|,
\]
instead of $\Phi_{j}(x)$, since $\nabla \tilde\Phi_{j}(x)=2Dg_{j}(x)^\T u_1(g_{j}(x))$ for all $x$ close enough to $\xb$ and ${j}\in I_B(\xb)$. As mentioned in~\cite[Remark 5.1-c]{crcq-naive}, using $\Phi_{j}$ or $\tilde\Phi_{j}$ characterize different approaches. Assuming the second type of naive approach, we recall that naive-CRCQ (respectively, naive-CPLD) is satisfied at $\xb\in \F$ when there exists a neighborhood $\mathcal{V}$ of $\xb$ such that, for every $J_B\subseteq I_B(\xb)$, the following holds: if the family~\eqref{socp:setndgclass} is  $\R^{|I_B(\xb)|}\times \prod_{{j}\in I_0(\xb)}\R^{m_{j}}$-linearly dependent (respectively, $\R^{|I_B(\xb)|}_+\times \prod_{{j}\in I_0(\xb)}\Lor_{m_{j}}$-linearly dependent), then \jo{the family }$\{Dg_{j}(x)^\T  u_1(g_{j}(x))\}_{j\in J_B}$ remains  linearly dependent for all $x$ in $\mathcal{V}$. Note that this definition coincides with nondegeneracy (respectively, Robinson's CQ) when no constraints are reducible -- that is, when $I_B(\xb)=\emptyset$ -- because $\emptyset$ is linearly independent. On the other hand, when all constraints are reducible, then Definition~\ref{socp:crcq} coincides with naive-CRCQ/CPLD. Thus, in the general case, both CQs of Definition~\ref{socp:crcq} are strictly weaker than their ``naive'' counterparts.
\end{remark}

\section{Stronger Constant Rank Conditions With Applications}\label{sec:appl}

As we already mentioned, our study of constraint qualifications is driven towards global convergence of algorithms for solving~\eqref{NSOCP}. In particular, we presented in the previous section a global convergence proof for the external penalty method under weak-CPLD; to extend this result for a broader class of iterative methods, we now introduce more robust adaptations of \sdpwcpld{} and \sdpwcrcq{}. This is similar to what we did in~\cite{seq-crcq-sdp} for NSDP problems. We start this section with an analogue of~\cite[Definition 4.2]{seq-crcq-sdp} in NSOCP, which characterizes a perturbed version of \socpwcrcq{} and \socpwcpld{}.

\begin{definition}[\socpscrcq{} and \socpscpld{}]\label{socp:scrcq}
	We say that $\xb\in \F$ satisfies the: 
	
	\begin{itemize}
		\item  \emph{Sequential CRCQ condition for NSOCP} (\socpscrcq{}) if for all sequences $\seq{x}\to \xb$ and \ima{\linebreak} $\seq{\Delta_{j}}\subseteq  \R^{m_{{j}}}$, ${j}\in I_0(\xb)\cup I_B(\xb)$, such that $\Delta^k_{j}\to 0$ for every ${j}$, there exists some $I\subseteq_{\infty} \N$, and convergent eigenvector sequences $\{u_1(g_{j}(x^k)+\Delta^k_{j})\}_{k\in I}\to \frac{1}{2}(1,-\bar{w}_{j})$ and $\{u_2(g_{j}(x^k)+\Delta^k_{j})\}_{k\in I}\to \frac{1}{2}(1,\bar{w}_{j})$, with $\|\bar{w}_{j}\| = 1$, for all ${j}\in I_0(\xb)$, such that for all subsets $J_B\subseteq I_B(\xb)$ and $J_{-},J_{+}\subseteq I_0(\xb)$, we have that: if the family of vectors $\mathcal{D}_{J_B,J_-,J_+}(\xb,\bar{w})$ is linearly dependent, \si{then $\mathcal{D}_{J_B,J_-,J_+}(x^k,w^k)$} \jo{then \linebreak $\mathcal{D}_{J_B,J_-,J_+}(x^k,w^k)$} remains linearly dependent for \si{every} \jo{all} $k\in I$ large enough, where $\bar{w}=[\bar{w}_j]_{j\in J_-\cup J_+}$ and $w^k=[w^k_j]_{j\in J_-\cup J_+}$ with
\begin{equation}\label{socp:setJ}
u_1(g_j(x^k)+\Delta_j^k)=\frac{1}{2}(1,-w^k_j) 
\quad\textnormal{and}\quad u_2(g_j(x^k)+\Delta_j^k)=\frac{1}{2}(1,w^k_j)
\end{equation}
for each $j\in J_-\cup J_+$. Recall that $\mathcal{D}_{J_B,J_-,J_+}(x,w)$ was defined in~\eqref{socp:familyd}.
	\item \emph{Sequential CPLD condition for NSOCP} (\socpscpld{}) if for all sequences $\seq{x}\to \xb$ and \ima{\linebreak} $\seq{\Delta_{j}}\subseteq \R^{m_{j}}$, ${j}\in I_0(\xb)\cup I_B(\xb)$, such that $\Delta^k_{j}\to 0$ for every ${j}$, there exists some $I\subseteq_{\infty} \N$, and convergent eigenvector sequences $\{u_1(g_{j}(x^k)+\Delta^k_{j})\}_{k\in I}\to \frac{1}{2}(1,-\bar{w}_{j})$ and $\{u_2(g_{j}(x^k)+\Delta^k_{j})\}_{k\in I}\to \frac{1}{2}(1,\bar{w}_{j})$, with $\|\bar{w}_{j}\| = 1$, for all ${j}\in I_0(\xb)$, such that for all subsets $J_B\subseteq I_B(\xb)$ and $J_{-},J_{+}\subseteq I_0(\xb)$, we have that: if $\mathcal{D}_{J_B,J_-,J_+}(\xb,\bar{w})$ is positively linearly dependent, then $\mathcal{D}_{J_B,J_-,J_+}(x^k,w^k)$ remains linearly dependent for all $k\in I$ large enough, where $\bar{w}$ and $w^k$ are as the previous item.
	\end{itemize}			
	
\end{definition}

Note that the nondegeneracy condition (as in Proposition~\ref{socp:conicli}) implies \socpscrcq{}, whereas Robinson's CQ implies \socpscpld{}. Moreover, these implications are strict, as it is shown in the next counterexample:

\begin{example}(Nondegeneracy and Robinson's CQ are strictly stronger than seq-CRCQ and seq-CPLD, respectively)\label{socp:excrcqndg}
Consider the constraint
\[
	g(x)\doteq (-x,x)\in \Lor_2
\]
at the point $\xb\doteq 0$, which is the only feasible point of the problem. In order to verify that $\xb$ satisfies \socpscpld{} and \socpscrcq{}, let $\seq{x}\to \xb$ and $\seq{\Delta}\to 0$ be arbitrary sequences. We will assume that there is some $I\subseteq_{\infty}\N$ such that $\widehat{g}(x^k)+\widehat{\Delta}^k>0$ for all $k\in I$, where $\Delta^k\doteq (\Delta^k_0, \widehat{\Delta}^k)\in \R^2$, since the other cases (as in Example~\ref{socp:exwndgnotndg2}) follow analogously. Then, we have
\[
	u_1(g(x^k)+\Delta^k)=\frac{1}{2}(1,-1)\quad \mand \quad u_2(g(x^k)+\Delta^k)=\frac{1}{2}(1,1),
\] 
which implies that $w^k=\bar{w}=1$ for all $k\in I$. Hence, the vectors \si{$Dg(\xb)^\T (1,-\bar{w})=-2$} \jo{$Dg(\xb)^\T \linebreak (1,-\bar{w})=-2$} and $Dg(x^k)^\T (1,w^k)=0$ are (positively) linearly dependent, but since $Dg(x^k)^\T (1,-w^k)=-2$ and $Dg(x^k)^\T (1,w^k)=0$ are also linearly dependent for every $k\in I$, we see that \socpscpld{} and \socpscrcq{} both hold, while Robinson's CQ and nondegeneracy do not.
\end{example}
Example~\ref{socp:excrcqndg} shows that \socpscrcq{} does not imply Robinson's CQ, and the converse is also false; otherwise Robinson's CQ would imply \socpwcrcq{}, contradicting Example~\ref{socp:exrobcrcq}. Further, note that Definition~\ref{socp:scrcq} is basically Definition~\ref{socp:crcq} with the addition of some perturbation sequences $\seq{\Delta_{j}}$. Then, \socpscpld{} implies \socpwcpld{}, and \socpscrcq{} implies \socpwcrcq{}, implying \textit{a fortiori} that \socpscpld{} and  \socpscrcq{} are constraint qualifications. However, the next example shows that these implications are both strict.

\begin{example}[Seq-CRCQ and seq-CPLD are stronger than weak-CRCQ and weak-CPLD, respectively]
Consider the constraint
\[
	g(x)\doteq (x^2,x,0)\in \Lor_3
\]
at $\bar{x}\doteq 0$. Let us begin by showing that $\xb$ satisfies both \socpwcrcq{} and \socpwcpld{}, so let $\seq{x}\to \xb$ be an arbitrary sequence. Again, as in Example~\ref{socp:exwndgnotndg2}, we will assume without loss of generality that there exists some $I\subseteq_{\infty} \N$ such that $x^k>0$ for every $k\in I$. In this case, we must have
\[
	u_1(g(x^k))=\frac{1}{2}\left(1, -1, 0\right) \quad \mand \quad u_2(g(x))=\frac{1}{2}\left(1, 1, 0\right),
\]
which yields $w^k=\bar{w}=(1,0)$ for every $k\in I$. Then, $Dg(\xb)^\T (1,-\bar{w})=-1$ and $Dg(\xb)^\T (1,\bar{w})=1$ are (positively) linearly dependent, but since \si{$Dg(x^k)^\T (1,-w^k)=2x^k-1$} \jo{$Dg(x^k)^\T \linebreak (1,-w^k)=2x^k-1$} and $Dg(x)^\T u_2(g(x))=2x^k+1$ are also linearly dependent for all $k\in I$ large enough so that $x^k\in (-\frac{1}{2},\frac{1}{2})$, it means that \socpwcrcq{} and \socpwcpld{} both hold at $\xb$.
	
However, taking any sequence $\seq{x}\to \xb$ such that $x^k>0$ for every $k\in \N$, and the perturbation vector
\[
	\Delta^k\doteq (-(x^k)^2,-x^k,x^k)\to 0,
\]
we have that $g(x^k)+\Delta^k\doteq (0,0,x^k)$, so its eigenvectors are uniquely determined by
\[
	u_1(g(x^k)+\Delta^k)=\frac{1}{2}\left(1, 0, -1\right) \quad \mand \quad u_2(g(x^k)+\Delta^k)=\frac{1}{2}\left(1, 0, 1\right),
\]
implying $Dg(x^k)^\T u_1(g(x^k)+\Delta^k)=2x^k>0$ and $Dg(x^k)^\T u_2(g(x^k)+\Delta^k)=2x^k>0$ are positively linearly independent for every $k\in \N$. But since $Dg(\xb)^\T (1,0,-1)=Dg(\xb)^\T (1,0,1)=0$ we conclude that \socpscpld{} and, by extension, \socpscrcq{}, both fail at $\xb$.
\end{example}

Furthermore, conditions~\socpscrcq{} and~\socpscpld{} can also be characterized in terms of a neighborhood, without sequences, just as the original CRCQ and CPLD conditions from NLP. Let us prove this:

\begin{proposition}\label{socp:nonseqscrcq}
Let $\xb\in\F$. Condition \socpscrcq{} (respectively, \socpscpld{}) holds at $\xb$ if, and only if, for every $\bar{w}\doteq [\bar{w}_{j}]_{{j}\in I_0(\xb)}$ with $\|\bar{w}_{j}\|=1, \ {j}\in I_0(\xb)$, there exists a neighborhood $\mathcal{V}$ of $(\xb,\bar{w})$ such that: for every $J_B\subseteq I_B(\xb)$ and $J_-,J_+\subseteq I_0(\xb)$, if $\mathcal{D}_{J_B,J_-, J_+}(\xb,\bar{w})$ is (positively) linearly dependent, then $\mathcal{D}_{J_B,J_-, J_+}(x,w)$ remains linearly dependent for every $(x,w)\in \mathcal{V}$ with $w\doteq [w_{j}]_{{j}\in I_0(\xb)}$ and $\|w_j\|=1$ for every $j\in J_-\cup J_+$. Here, $\mathcal{D}_{J_B,J_-, J_+}(x,w)$ is as defined in~\eqref{socp:familyd}.
\end{proposition}

\begin{proof}
Suppose that there exists some subsets $J_B\subseteq I_B(\xb)$ and $J_-,J_+\subseteq I_0(\xb)$, and some $\bar{w}=[\bar{w}_{j}]_{{j}\in J_-\cup J_+}$ such that $\mathcal{D}_{J_B,J_-,J_+}(\xb,\bar{w})$ is (positively) linearly dependent, but there is a sequence $\{(x^k,w^k)\}_{k\in \N}\to (\xb,\bar{w})$ with $w^k\doteq [w_{j}^k]_{{j}\in J_-\cup J_+}$ and $\|w^k_j\|=1$, such that $\mathcal{D}_{J_B,J_-,J_+}(x^k,w^k)$ is linearly independent for all $k\in\N$. Define, for each $k\in \N$ and ${j}\in J_B\cup I_-\cup I_+$, the perturbation vector
\begin{equation}\label{eq:deltabom}
	\Delta^k_{j}\doteq 
	\left\{
		\begin{aligned}
			\frac{1}{k}\left(1, w^k_{j}\right)-g_{j}(x^k), & \quad \textnormal{ if } {j}\in J_-\cup J_+\\
			g_{j,0}(\bar{x}) \left(1, \frac{\widehat{g}_j(x^k)}{\|\widehat{g}_j(x^k)\|}\right)-g_{j}(x^k), & \quad \textnormal{ if } {j}\in J_B,
		\end{aligned}
	\right.
\end{equation}
which implies that $g_{j}(x^k)+\Delta^k_{j}\in \bd_+ \Lor_{m_{j}}$ and hence its eigenvectors are uniquely determined for every such $j$ and $k$. This contradicts Definition~\ref{socp:scrcq}.

Conversely, pick any sequences $\seq{x}\to \xb$ and $\seq{\Delta_{j}}\to 0$, ${j}\in I_0(\xb)\cup I_B(\xb)$, and any subsets $J_B\subseteq I_B(\xb)$ and $J_-,J_+\subseteq I_0(\xb)$. Then, define $\seq{w}$ as in Definition~\ref{socp:scrcq} and let $\bar{w}=[\bar{w}_j]_{j\in J_-\cup J_+}$ be such that $\|\bar{w}_j\|=1$ for every $j\in J_-\cup J_+$ and $\lim_{k\in I} u_1(g_{j}(x^k)+\Delta^k_{j})= \frac{1}{2}(1,-\bar{w}_{j})$ and $\lim_{k\in I} u_2(g_{j}(x^k)+\Delta^k_{j})= \frac{1}{2}(1,\bar{w}_{j})$, for some $I\subseteq_{\infty}\N$. Note that $\lim_{k\in I} w^k= \bar{w}$, so if $\mathcal{D}_{J_B,J_-,J_+}(\xb,\bar{w})$ is (positively) linearly dependent, then $\mathcal{D}_{J_B,J_-,J_+}(x^k,w^k)$ is remains linearly dependent for every $k$ large enough. 
\jo{\qed} \end{proof}

\begin{remark}\label{rem:slices}
Note that Proposition~\ref{socp:nonseqscrcq} reveals that Definition~\ref{socp:scrcq} characterizes a ``constant rank condition, or constant (positive) linear dependence, by conical slices''. For example, consider a single constraint $g(x)\in \Lor_m$ at a point $\xb$ such that $g(\xb)\in \Lor_m$; then, \socpscrcq{} holds at $\xb$ if, and only if, for each conical slice of $\Lor_m$, which can be of two types:
\begin{enumerate}
\item $C_{\bar{w}}^1=\textnormal{cone}(\{(1,\bar{w})\})$, for some $\bar{w}\in \R^{m-1}$ such that $\|\bar{w}\|=1$;

\item $C_{\bar{w}}^2=\textnormal{cone}(\{(1,-\bar{w}),(1,\bar{w})\})$, for some $\bar{w}\in \R^{m-1}$ such that $\|\bar{w}\|=1$;
\end{enumerate}
the dimension of
\[
	Dg(x)^\T \textnormal{span}(C_w^i)=\left\{
	\begin{array}{ll}
		\textnormal{span}(\{Dg(x)^\T(1,w)\}), & \textnormal{if } i=1,\\
		\textnormal{span}(\{Dg(x)^\T (1,-w),Dg(x)^\T(1,w)\}), & \textnormal{if } i=2,
	\end{array}			
	\right.
\]
remains constant for every $(x,w)$ close enough to $(\xb,\bar{w})$. The \socpscpld{} condition admits a similar phrasing. That is, the local constant rank property must hold for every perturbation of $\xb$ and every perturbation of the slice as well, roughly speaking, and the existence of two types of conical slices describes, intuitively, why should one consider every subset of \si{$\{Dg(x)^\T (1,-w),Dg(x)^\T(1,w)\}$} \jo{$\{Dg(x)^\T (1,-w), \linebreak Dg(x)^\T(1,w)\}$}.  
\end{remark}

\subsection{Global Convergence of Algorithms With Some Examples}\label{sec:algorithms}

Here, we show that the condition \socpscpld{} can be used to prove global convergence of an abstract class of iterative algorithms, namely the ones that generate sequences of approximate solutions $\seq{x}$, which we will assume to be convergent to some $\xb$, and approximate Lagrange multipliers $\seq{\mu_{j}}\subseteq \Lor_{m_{j}}$, ${j}\in\{1,\ldots,q\}$, in the sense that
\begin{equation}\label{eq:akkt1}
	\nabla_x L(x^k,\mu^k_1,\ldots,\mu^k_q)\to 0
\end{equation}
and for every $k\in \N$,
\begin{equation}\label{eq:akkt2}
	g_{j}(x^k)+\Delta^k_{j}\in \Lor_{m_{j}} \quad \textnormal{and}\quad \langle g_{j}(x^k)+\Delta^k_{j}, \mu_{j}^k \rangle=0
\end{equation}
for some sequences $\Delta^k_{j}\to 0$, ${j}\in \{1,\ldots,q\}$. Later in this section, we will discuss some details about some popular algorithms that generate this kind of sequence. But first, let us prove our unified global convergence result:

\begin{theorem}[Global convergence under seq-CPLD]\label{socp:safeguard}
Let $\seq{x}$ and $\seq{\mu_{j}}\subseteq \Lor_{m_{j}}$, ${j}\in\{1,\ldots,q\}$ satisfy~\eqref{eq:akkt1} and~\eqref{eq:akkt2}, and let $\xb$ be a feasible limit point of $\seq{x}$ that satisfies \socpscpld{}. Then, $\xb$ satisfies the KKT conditions.
\end{theorem}
\begin{proof}
For simplicity, let us assume that $\seq{x}\to \xb$. From~\eqref{eq:akkt1} we obtain that
\begin{equation}\label{safe:kkt}
\nabla f(x^k)-\sum_{{j}=1}^q Dg_{j}(x^k)^\T \mu^k_{j}\to 0.
\end{equation}

Now, by~\eqref{eq:akkt2} we obtain
\[
	\mu_j^k=\left\{
		\begin{array}{ll}
			0, & \textnormal{if } g_j(x^k)+\Delta_j^k\in \int \Lor_{m_j},\\
			\frac{\mu_{j,0}^k}{g_{j,0}(x^k)+\Delta_{j,0}^k} \Gamma_j (g_j(x^k)+\Delta_j^k), & \textnormal{if } g_j(x^k)+\Delta_j^k\in \bd^+ \Lor_{m_j},
		\end{array}
	\right.
\]
where $\Gamma_j$ is defined in~\eqref{eq:gamma}, and $\mu_j^k$ can be any point of $\Lor_{m_j}$ if $g_j(x^k)+\Delta_j^k=0$. Thus, there exists a spectral decomposition of
\[
	\mu_{{j}}^{k} \doteq \alpha_{{j}}^{k}u_{1}(\mu_j^k) + \beta_{{j}}^{k}u_{2}(\mu_j^k),
\]
such that $u_{1}(\mu_j^k)$ and $u_{2}(\mu_j^k)$ are also eigenvectors of $g_{j}(x^k)+\Delta_{j}^k$ for every $k\in\N$. Moreover, note that~\eqref{eq:akkt2} implies that $\alpha_j^k\lambda_{1}(g_j(x_j^k)+\Delta_j^k)=0$ and $\beta_j^k\lambda_{2}(g_j(x_j^k)+\Delta_j^k)=0$ for every $k\in \N$ and every $j\in \{1,\ldots,q\}$. Then $\beta_{{j}}^k=0$ for all $k$ large enough and for every ${j}\in I_B(\xb)\cup I_{\int}(\xb)$, because $\lambda_{2}(g_j(x_j^k)+\Delta_j^k)>0$ for all large $k$ in these cases. Therefore, we can rewrite~\eqref{safe:kkt} as 
\si{
\begin{displaymath}
\nabla f(x^k)-\sum_{{j}\in I_{0}(\xb)} \left(\alpha_{{j}}^{k} Dg_{{j}}(x^{k})^{\T} u_{1}(\mu_j^k) + \beta_{{j}}^{k} Dg_{{j}}(x^{k})^{\T} u_{2}(\mu_j^k)\right) - \sum_{{j}\in I_{B}(\xb)} \alpha_{{j}}^{k} Dg_{{j}}(x^{k})^{\T}u_{1}(\mu_j^k)\to 0.
\end{displaymath}
}
\jo{
\begin{eqnarray*}
\nabla f(x^k) &-& \sum_{{j}\in I_{0}(\xb)} \left(\alpha_{{j}}^{k} Dg_{{j}}(x^{k})^{\T} u_{1}(\mu_j^k) + \beta_{{j}}^{k} Dg_{{j}}(x^{k})^{\T} u_{2}(\mu_j^k)\right)\\
 &-& \sum_{{j}\in I_{B}(\xb)} \alpha_{{j}}^{k} Dg_{{j}}(x^{k})^{\T}u_{1}(\mu_j^k)\to 0.
\end{eqnarray*}
}
The rest of the proof is similar to the proof of Theorem~\ref{socp:mincrcqkkt}, which consists of using Carathéodory's Lemma in the above relation, assuming that the new scalars are unbounded, and then directly applying Definition~\ref{socp:scrcq} to reach a contradiction, hence it shall be omitted.
\jo{\qed} \end{proof}

The sequences satisfying~\eqref{eq:akkt1} and~\eqref{eq:akkt2} are known as \textit{Approximate-KKT} (AKKT) sequences, which define a sequential optimality condition introduced by \si{Andreani et al. in}\jo{Andreani et al. in}~\cite{Santos2019} for NSOCP problems.  Also, we must mention that several algorithms generate AKKT sequences; one recurrent example (see~\cite[Algorithm 5.1]{Santos2019}) is the classical Hestenes-Powell-Rockafellar \ima{(see} \cite{hestenes,powell,Rocka1974}\ima) augmented Lagrangian method, which is based on the perturbed penalty function
\[
	L_{\rho,\tilde{\mu}_1,\ldots,\tilde{\mu}_q}(x)\doteq f(x)+\frac{\rho}{2}\left[\sum_{j=1}^q \left\|\mathcal{P}_{\Lor_{m_j}}\left(-g_j(x)-\frac{\tilde{\mu}_j}{\rho}\right)\right\|^2 - \left\|\frac{\tilde{\mu}_j}{\rho}\right\|^2\right],
\]
where $\rho\in \R_+$ and $\tilde{\mu}_j\in \Lor_{m_j}$, $j\in\{1,\ldots,q\}$, are given parameters. The sequence $\seq{x}$ is computed as approximate stationary points of $L_{\rho_k,\tilde{\mu}_1^k,\ldots,\tilde{\mu}_q^k}(x)$ and their associate approximate Lagrange multipliers are given by
\[
	\mu_{j}^k\doteq \mathcal{P}_{\Lor_{m_{j}}}\left(-\rho_k g_{j}(x^k)-\tilde{\mu}_{{j}}^k\right)
\]
where $\{\rho_k\}_{k\in\N}$ is the penalty parameter and $\seq{\tilde{\mu}_{j}}\subseteq \Lor_{m_{j}}$ are given sequences and $\Delta^k_j\doteq \frac{\mu^k_j-\tilde{\mu}^k_j}{\rho_k}$ for every $j\in \{1,\ldots,q\}$. In particular, note that $\nabla L_{\rho_k,\tilde{\mu}_1^k,\ldots,\tilde{\mu}_q^k}(x^k)=\nabla_x L(x^k,\mu_1^k,\ldots,\mu_q^k)$ for every $k\in \N$. See also~\cite{Andreani2020} for a more detailed discussion on this topic.

Besides the augmented Lagrangian and its variants, the \textit{sequential quadratic programming} (SQP) algorithm of \si{Kato and Fukushima}\jo{Kato and Fukushima}~\cite[Algorithm 1]{katofukusqp} can also be proved to generate output sequences that satisfy~\eqref{eq:akkt1} and~\eqref{eq:akkt2}. For completeness, we state their algorithm below:

\begin{center}
\centering
\begin{minipage}{\linewidth}
\begin{algorithm}[H]
	\caption{Sequential quadratic programming algorithm of~\cite{katofukusqp}.}
	\label{sqp}
	\medskip

    {\it \textbf{Input}:} An initial point $x^{0}\in \R^n$ and some parameters $\alpha_0>0$, $\sigma\in(0,1)$, $\gamma_1>0$, $\gamma_2>0$, and $\tau>0$.

	\medskip
	
	Set $k \doteq 0$. Then:
	
	\medskip
	{\it \textbf{Step 1}:} Choose a symmetric positive definite matrix $M^k\in \R^{n\times n}$ such that $\gamma_1\|z\|^2\leq z^\T M^k z\leq \gamma_2\|z\|^2$ for every $z\in \R^n$, and find a solution $d^k$ if possible of the problem:
		\begin{equation}
  \tag{QP}
  \begin{aligned}
    & \underset{d \in \mathbb{R}^{n}}{\text{Minimize}}
    & & \nabla f(x^k)^\T d + \frac{1}{2}d^\T M^k d, \\
    & \text{subject to}
    & & g_j(x^k)+Dg_j(x^k)d \in \Lor_{m_j}, \ \forall j\in \{1,\ldots,q\}
  \end{aligned}
  \label{qp}
\end{equation}
 together with its Lagrange multipliers $\mu^{k}_j\in \Lor_{m_j}$, $j\in \{1,\ldots,q\}$; if $d^k=0$, then \textbf{stop};

	\medskip
	{\it \textbf{Step 2}:}
	Set the penalty parameter as follows: If $\alpha^k\geq \max\{|\mu^k_{j,0}|\colon j\in \{1,\ldots,q\}\}$, then $\alpha^{k+1}\doteq \alpha^k$; otherwise, $\alpha^{k+1}\doteq \max\{\alpha^k, |\mu^k_{j,0}|\colon j\in \{1,\ldots,q\}\}+\tau$;
	
	\medskip
	{\it \textbf{Step 3}:} Compute some scalar $t^k\in (0,1]$ satisfying
	\begin{equation}\label{eq:sqpstepsize}
		\Phi_{\alpha^{k+1}}(x^k)-\Phi_{\alpha^{k+1}}(x^k+t^k d^k)\leq \sigma t^k (d^k)^\T M^k d^k;
	\end{equation}
	where
	\[
		\Phi_{\alpha}(x)\doteq f(x) + \alpha\sum_{j=1}^q \max\{0,-g_{j,0}(x)-\|\widehat{g}_j(x)\|\}
	\]
	is a penalty function;

	\medskip
	{\it \textbf{Step 4}:} Set $x^{k+1}\doteq x^k + t^k d^k$ and $k \doteq k+1$, and go to Step 1.
	
	\medskip
\end{algorithm}
\end{minipage}
\end{center}

\

In~\cite{katofukusqp}, \si{Kato and Fukushima}\jo{Kato and Fukushima}\ima{the authors} proved the global convergence of Algorithm~\ref{sqp} under the following assumptions:
\begin{itemize}
\item [A1.] Step 1 is well-defined for every $k\in \N$;
\item [A2.] The output sequence $\{x^k\}_{k\in \N}$ of Algorithm~\ref{sqp} is bounded;
\item [A3.] The multiplier sequences $\{\mu_j^k\}_{k\in \N}$, $j\in \{1,\ldots,q\}$ computed by the method are all bounded.
\end{itemize}

Observe that these assumptions, although somewhat standard, are demands over the behavior of the algorithm itself instead of the problem, and a convergence theory that makes strong assumptions over the behavior of the method is, to say the best, fragile. Even so, A1 and A2 can be considered a ``necessary evil'' since their violation means that the execution of the method has terminated in failure. Assumption A3, on the other hand, is not plausible since it basically guides the method towards convergence. Instead of A3, an assumption over the problem (and not the method), for instance the fulfilment of a constraint qualification at every limit point of $\seq{x}$, would be more reasonable for illustrating its strength. Of course Robinson's CQ is well-suited for this role since it implies A3, but an improvement can be made with the weaker constraint qualification \socpscpld{}; that is, under the following assumption:

\begin{itemize}
\item [A4.] All limit points of $\{x^k\}_{k\in \N}$ satisfy \socpscpld{}.
\end{itemize}

Then, we can easily rephrase an excerpt from the proof of~\cite[Theorem 1]{katofukusqp} and apply Theorem~\ref{socp:safeguard} to obtain the same convergence result of~\cite{katofukusqp} under A1, A2, and A4, instead of A3 or Robinson's CQ. However, it should be noticed that A4 may hold even when the approximate Lagrange multiplier sequences are unbounded.

\begin{proposition}
Under A1, the output sequences $\seq{x}$ and $\seq{\mu_j}$, $j\in\{1,\ldots,q\}$, of Algorithm~\ref{sqp} satisfy~\eqref{eq:akkt1} and~\eqref{eq:akkt2}.
\end{proposition}

\begin{proof}
For each $k\in \N$, assumption A1 tells us that $x^k$ and $\mu_j^k\in \Lor_{m_j}$, $j\in \{1,\ldots,q\}$ satisfy the following:
\begin{eqnarray*}
	\nabla f(x^k)+M^kd^k - \sum_{j=1}^q Dg_j(x^k)^\T \mu_j^k=0,\label{eq:linkkt1} & \\
	\langle \mu_j^k, g_j(x^k)+Dg_j(x^k) d^k\rangle=0, & \forall j\in\{1,\ldots,q\},\label{eq:linkkt2}\\
	g_j(x^k)+Dg_j(x^k) d^k\in \Lor_{m_j}, & \forall j\in\{1,\ldots,q\}.\label{eq:linkkt3}
\end{eqnarray*}
Since by construction $\seq{M}$ is bounded and by~\cite[Theorem 1]{katofukusqp} we have $\seq{d}\to 0$, the conclusion follows by taking $\Delta^k_j\doteq Dg_j(x^k)d^k$ for every $k\in \N$ and every $j\in\{1,\ldots,q\}$.
\jo{\qed} \end{proof}

For the sake of completeness, we present a formal statement of the convergence result of Algorithm~\ref{sqp} under~\socpscpld{}, which follows immediately from the previous proposition.

\begin{corollary}
Assume A1, A2, and A4. Every limit point of the sequence $\seq{x}$ generated by Algorithm~\ref{sqp} satisfies the KKT conditions.
\end{corollary}

\subsection{On Error Bounds and Robustness}

Another interesting implication of CRCQ and CPLD from the literature concerns error bounds. To address it to NSOCP, let us recall the definition of the so-called \textit{metric subregularity CQ} for (\ref{NSOCP}) problems.

\begin{definition}[MSCQ]
Let $\xb$ be a feasible point of (\ref{NSOCP}) and let \si{$g(x) \doteq (g_{1}(x), \ldots, g_{q}(x))$} \jo{$g(x) \doteq (g_{1}(x), \linebreak \ldots, g_{q}(x))$}. We say that $\xb$ satisfies the \textnormal{metric subregularity CQ} (MSCQ) when there exists some $\gamma > 0$ and a neighborhood $\mathcal{V}$ of $\xb$ such that
\begin{displaymath}
\textnormal{dist}(x,\F)\leq \gamma \textnormal{dist}(g(x),\Pi_{{j}=1}^{q} \Lor_{m_{{j}}})
\end{displaymath}
for every $x \in \mathcal{V}$, where $\F$ is the feasible set of~\eqref{NSOCP}. 
\end{definition}

The following result shows a sufficient condition in order to obtain MSCQ. This result is an adaptation from \si{Minchenko and Stakhovski~\cite[Theorem 2]{rcr}}\jo{Minchenko and Stakhovski~\cite[Theorem 2]{rcr}}\ima{~\cite[Theorem 2]{rcr}} for nonlinear programming problems. Also, an extension for semidefinite programming was made in~\cite[Proposition 5.1]{seq-crcq-sdp} and hence its proof will be omitted.

\begin{proposition}\label{socp:metric_prop}
Let $\xb\in \F$ and assume that $g_{{j}}$ are twice differentiable around $\xb$, with ${j}\in\{1,\ldots,q\}$. Given $x\in \R^{n}$, let $\Lambda_{x}(y)$ denote the set of Lagrange multipliers associated with any given solution $y$ of the problem of minimizing $\|z-x\|$ subject to $g_{{j}}(z)\in \Lor_{m_{{j}}}$, ${j}\in \{1,\ldots,q\}$, $z\in \R^{n}$. If there exist numbers $\tau>0$ and $\delta>0$ such that $\Lambda_{x}(y)\cap \cl(B(0,\tau))
\neq \emptyset$ for every $x\in B(\xb,\delta)$, then $\xb$ satisfies MSCQ. 
\end{proposition}
Then, we shall prove that \socpscpld{} and \socpscrcq{} are robust, and this, together with Proposition~\ref{socp:metric_prop}, is enough to show that that they imply MSCQ.

\begin{theorem}[Robustness of seq-CPLD (and seq-CRCQ)]
	If $\xb\in \F$ satisfies \socpscpld{} (or \socpscrcq{}), then:
	\begin{enumerate}
	\item there is a neighborhood $\mathcal{V}$ of $\xb$, such that every $x\in \mathcal{V}\cap\F$ also satisfies \socpscpld{} (respectively, \socpscrcq{});
	\item MSCQ holds at $\xb$.
\end{enumerate}
\end{theorem}

\begin{proof}
We will only exhibit the proof for~\socpscpld{}, since the proof for~\socpscrcq{} is analogous. Suppose that item 1 is false, then there is a sequence $\seq{x}\to \xb$ such that \socpscpld{} fails at $x^k$, for all $k\in \N$. That is, for each $k\in \N$ there is some $w^k\doteq [w^k_j]_{j\in I_0(x^k)}$ with $\|w_k^k\|=1$ for every $j\in I_0(x^k)$, some sequences $\{x^k_\ell\}_{\ell\in \N}\to x^k$ and $\{w_\ell^k\}_{\ell\in \N}\to w^k$, and subsets $J_B^k\subseteq I_B(x^k)$ and $J_-^k,J_+^k\subseteq I_0(x^k)$ such that $\mathcal{D}_{J_B^k,J_-^k,J_+^k}(x^k,w^k)$ is positively linearly dependent, but $\mathcal{D}_{J_B^k,J_-^k,J_+^k}(x^k_\ell, w_\ell^k)$ is linearly independent for every $\ell\in \N$. By the infinite pigeonhole principle, we can assume that $I_0=I_0(x^k)$ and $I_B=I_B(x^k)$ are the same for every $k\in \N$, and also that $J_B=J_B^k$, $J_-=J_-^k$, and $J_+=J_+^k$ for every $k\in \N$, passing to a subsequence if necessary. Moreover, note that we can also assume that $I_0\subseteq I_0(\xb)$ and $I_B\subseteq I_0(\xb)\cup I_B(\xb)$. Now consider the following sets:
\[
	\tilde{J}_B\doteq J_B\cap I_B(\xb), \quad \tilde{J}_-\doteq J_-\cup (J_B\cap I_0(\xb)), \quad \textnormal{and} \quad \tilde{J}_+\doteq J_+.
\]
By construction, note that $\mathcal{D}_{\tilde{J}_B,\tilde{J}_-,\tilde{J}_+}(x^k_\ell, w_\ell^k)$ is linearly independent for every $k,\ell\in \N$. For each $k$, let $\ell(k)$ be such that $\|w^k-w_{\ell(k)}^k\|<\frac{1}{k}$, and let $\bar{w}$ be any limit point of $\seq{w}$. Without loss of generality, we will assume that $w^k\to \bar{w}$, which also implies that $w^k_{\ell(k)}\to \bar{w}$.

Analogously to~\eqref{eq:deltabom}, we can construct some $\Delta^k_j\in \R^{m_j}$ for every $j\in I_0(\xb)\cup I_B(\xb)$, such that \ima{\linebreak} $g_j(x^k_{\ell(k)})+\Delta_j^k\in \bd_+\Lor_{m_j}$ and hence its eigenvectors are uniquely determined by
\[
u_1(g_j(x^k_{\ell(k)})+\Delta_j^k)=\frac{1}{2}\left(1,\frac{\widehat{g}_j(x^k_{\ell(k)})}{\|\widehat{g}_j(x^k_{\ell(k)})\|}\right), \ \forall j\in \tilde{J}_B,
\]
and
\[
u_1(g_j(x^k_{\ell(k)})+\Delta_j^k)=\frac{1}{2}\left( 1, -w^k_{\ell(k)} \right)
\]
and
\[
u_2(g_j(x^k_{\ell(k)})+\Delta_j^k)=\frac{1}{2}\left( 1, w^k_{\ell(k)} \right), \ \forall j\in \tilde{J}_-\cup \tilde{J}_+.
\] 
With this in mind, on the one hand, we have that $\mathcal{D}_{\tilde{J}_B,\tilde{J}_-,\tilde{J}_+}(\xb, \bar{w})$ is linearly dependent, because the family $\mathcal{D}_{\tilde{J}_B,\tilde{J}_-,\tilde{J}_+}(x^k, w^k)$ is linearly dependent for every $k\in \N$. But on the other hand, $\mathcal{D}_{\tilde{J}_B,\tilde{J}_-,\tilde{J}_+}(x^k_{\ell(k)}, w_{\ell(k)}^k)$ is linearly independent for every $k\in \N$, and the fact that the eigenvectors of $g_j(x^k_{\ell(k)})+\Delta_j^k$ are uniquely determined for all $j\in \tilde{J}_B \cup\tilde{J}_-\cup\tilde{J}_+$, together with $w^k_{\ell(k)}\to \bar{w}$, contradicts \socpscpld{} at $\xb$.

The proof of item 2 follows analogously to the proof of~\cite[Theorem 5.1]{seq-crcq-sdp}, which is essentially a corollary of item 1 and Proposition~\ref{socp:metric_prop}; hence it will be omitted.
\jo{\qed} \end{proof}

For a better exposition, what follows is a diagram that represents the relationship of some existing constraint qualifications and the ones that we present in this paper. 

\begin{figure}[!htb]
	\centering
	\begin{tikzpicture}
		\tikzset{
		    old/.style={rectangle, very thick, minimum size=0.7cm,fill=black!15,draw=black},
		    new/.style={rectangle,very thick, minimum size=0.7cm,fill=black!5,draw=black},
		    v/.style={-stealth,very thick,black!80!black},
		    u/.style={-stealth,dashed,very thick,black!80!black},
		}
		\node[old] (NDG) at (0,0) {Nondegeneracy};
		\node[old] (RCQ) at (3.5,0) {Robinson's CQ};
		\node[new] (SCR) at (3.5,-2) {Seq-CRCQ};
		\node[new] (SCP) at (7,0) {Seq-CPLD};
		\node[old] (MSCQ) at (9.5,0) {MSCQ};
		\node[new] (WCR) at (7,-3.8) {Weak-CRCQ};
		\node[new] (WCP) at (9.2,-2) {Weak-CPLD};
		\node[new] (WNDG) at (0,-2) {Weak-nondegeneracy};
		\node[new] (WROB) at (3.5,-3.8) {Weak-Robinson's CQ};
		\draw[v] (NDG) -- (RCQ);
		\draw[v] (NDG) -- (WNDG);
		\draw[v] (WNDG) -- (WROB);
		\draw[v] (NDG) -- (SCR);	
		\draw[v] (SCR) -- (SCP);
		\draw[v] (WROB) -- (6.7,-2) -- (WCP);
		\draw[v] (WCR) -- (WCP);
		\draw[v] (RCQ) -- (SCP);
		\draw[u] (SCP) -- (MSCQ);
		\draw[v] (SCR) -- (WCR);
		\draw[v] (SCP) -- (WCP);
		\draw[v] (WNDG) -- (0,-3.8) -- (2.2,-5) -- (4.8,-5) -- (WCR);
		\draw[u] (RCQ) -- (5.5,-2) -- (5.5,-3.8)  -- (WROB);
	\end{tikzpicture}
	\caption{Constraint qualifications for NSOCP. Strict implications are represented by solid arrows. Possibly two-sided implications are represented by dashed arrows.}
	\label{fig:relationsCQ}
\end{figure}
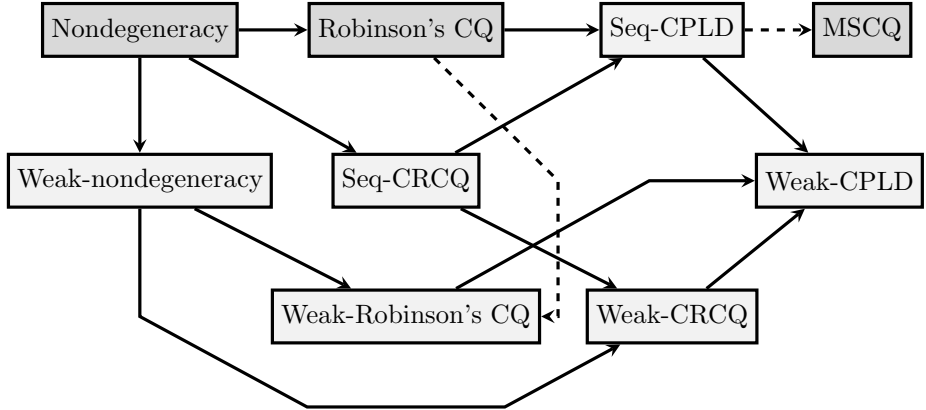

%
%
%
%
%
%
%
%
%
%
%
%
%
%

\section{Conclusion}\label{sec:conclusion}

In our previous work\ima{ (see}~\cite{weaksparsecq}\ima{)}, we studied two ways of incorporating some structural features of the semidefinite cone into the nondegeneracy condition of \si{Shapiro and Fan}\jo{Shapiro and Fan}~\cite{shapfan}; among them was the eigendecomposition, which has always been widely exploited in the design of algorithms for NSDP -- for instance, see \cite{kocvara}. Quite surprisingly, after incorporating eigendecompositions into the nondegeneracy condition (and also Robinson's CQ) we obtained a strictly weaker constraint qualification by means of considering only converging sequences of eigenvectors associated with a given point of interest, which was called weak-nondegeneracy (respectively, weak-Robinson's CQ). Moreover, this ``sequential approach'' allowed us to bypass the main difficulty in generalizing the celebrated constant rank constraint qualification of NLP, to NSDP~\cite{seq-crcq-sdp}, which is the presence of a potentially non-zero duality gap even in feasible linear problems (see also~\cite{facial-crcq} for a more detailed discussion on this topic). In this paper we bring those concepts to the context of NSOCP where several improvements with respect to the NSDP approach were made.

It is well known (see, for instance, the seminal work of \si{Alizadeh and Goldfarb}\jo{Alizadeh and Goldfarb}~\cite{surveysocp}) that although NSOCP problems can be reformulated as particular instances of NSDP problems, solving them via such a reformulation is generally not a good practice for a handful of reasons. Likewise, extensions of the sequential-type constraint qualifications of~\cite{seq-crcq-sdp,weaksparsecq} to NSOCP demand a specialized analysis to be properly conducted. In fact, the second-order cone induces a distinguished eigendecomposition that is easily computable, contrary to NSDP, which allows a deeper analysis to be made. For instance, besides extending the weak variants of the nondegeneracy condition and Robinson's CQ from NSDP to NSOCP, this paper also presents a full comparison between these weak conditions and their standard versions, which is an issue we could not properly address in~\cite{weaksparsecq}. Some technical results from \cite{weaksparsecq} could also be explained in a somewhat natural way in this paper. Moreover, besides extending the constant rank conditions from~\cite{seq-crcq-sdp}, we also gave them a geometrical interpretation in terms of the conical slices of the second-order cone (Remark~\ref{rem:slices}).

Very recently, we have been extending the notions of constant rank-type constraint qualifications to the contexts of NSDP and NSOCP. While \cite{facial-crcq} follows an implicit function approach pioneered by \si{Janin}\jo{Janin} \cite{janin} and giving rise to a definition of CRCQ that enjoys strong second-order properties, in this paper we exploit a sequential approach \ima{based on~} \cite{Santos2019}, which allows even weaker conditions to be defined, such as the CPLD condition, while enjoying global convergence properties of several algorithms without assuming boundedness of the set of Lagrange multipliers but still allowing computation of error bounds. Not surprisingly, when extending NLP concepts to the conic context, different points of view may give rise to different possible extensions, each one extending different applications of the concept. Some relevant topics in conic programming that we expect the conditions we define in this paper will be particularly relevant are: in the global convergence analysis of other classes of algorithms, including second-order algorithms\ima{ (see}~\cite{fukuda-haeser-mito}\ima{)}; the study of the boundedness of Lagrange multipliers estimates and the use of scaled stopping criteria \ima{ (as in }\cite{scaledal}\ima{)}; stability analysis of parametric optimization problems \ima{ (see, for instance, }\cite{bonnansramirez,outratagfrerer,param,stab2,stab1,fullstab,outrataramirez}\ima{)}; and necessary optimality conditions for some extended classes of bilevel optimization problems with conic constraints \ima{(}\cite{chenyezhangzhou,chiwanhao,ye}\ima{)}.

\jo{
\begin{acknowledgements}
\acknowledgementstext
\end{acknowledgements}
}

\ima{
\section*{Funding}
\acknowledgementstext
}

\si{\bibliographystyle{plain}}
\jo{\bibliographystyle{plain}}
\ima{\bibliographystyle{IMANUM-BIB}}

\end{document}